\documentclass{amsart}
\usepackage{graphicx,amssymb,amsthm,amsfonts,amsmath,enumitem,comment,pinlabel,hyperref,multirow,diagbox, float} 
\usepackage[usenames,dvipsnames]{xcolor}
\usepackage[textsize=tiny, textwidth=0.8in]{todonotes}
\newcommand\TK[1]{\color{ProcessBlue}#1 \color{black}}

\newcommand\IB[1]{\color{ForestGreen}#1 \color{black}}

\definecolor{Purple}{cmyk}{.5,0.75,0,0}
\definecolor{Green}{cmyk}{1,0,1,0}
\definecolor{Blue}{cmyk}{1,1,0,0}

\theoremstyle{plain}
\newtheorem{theorem}{Theorem}[section]

\newtheorem{obs}[theorem]{Observation}
\newtheorem{prop}[theorem]{Proposition}

\newtheorem{conjecture}[theorem]{Conjecture}

\newtheorem*{min genus algo}{Minimal genus algorithm}
 
\theoremstyle{definition}

\newtheorem{example}[theorem]{Example}

\theoremstyle{remark}

\title[Crosscap numbers via state codes]{Crosscap numbers of alternating links via state codes}
\author{Isaias Bahena, Thomas Kindred, and Jason Parsley}
\date{\today}

\begin{document}

\maketitle

\begin{abstract}
We describe a way of encoding a Kauffman state as a set of tuples, similar to a Gauss code.  Then we describe a procedure for using these state codes to determine the unoriented genus and crosscap number of any prime alternating knot or link. Finally, we compute these values for all such links through 14 crossings and all such knots through 19 crossings (this data is new for links with 10-14 crossings and knots with 14-19 crossings), and we identify several intriguing patterns in the resulting data.   \end{abstract}

\section{Introduction}

Every knot or link $L\subset S^3$ has many {\it spanning surfaces}: compact, connected surfaces $F$ embedded in $S^3$ with $\partial F=K$. Some are 1-sided, and some are 2-sided.  The {\it complexity} $\beta_1(F)$ of such a surface is the rank of its first homology group, or equivalently the number of cuts required to reduce $F$ to a disk, or the number of holes in $F$.  The {\it unoriented genus} $\Gamma(K)$ is the smallest complexity among all (1- or 2-sided) spanning surfaces for $K$, and the {\it crosscap number} $\gamma(K)$ is the minimal complexity among all 1-sided spanning surfaces for $K$:
\begin{align*}
\Gamma(K)&=\min\{\beta_1(F):~F\text{ is a spanning surface for }K\},\\
\gamma(K)&=\min\{\beta_1(F):~F\text{ is a 1-sided spanning surface for }K\}.
\end{align*}

Computing the unoriented genus and crosscap number of an arbitrary knot remains an open and challenging problem.  When $K$  has a nontrivial alternating diagram $D$, however, Adams and Kindred \cite{ak} prove that both $\Gamma(K)$ and $\gamma(K)$ are realized by {\it state surfaces} from $D$. These are spanning surfaces constructed from Kauffman states in a way akin to Seifert's algorithm. See \textsection\ref{S:Background} for details.

If $D$ has $c$ crossings and $F$ is a state surface from a Kauffman state with $s$ state circles, then $\beta_1(F)=1+c-s$. Thus, when $D$ is an alternating diagram of $K$ the problem of computing $\Gamma(K)$ reduces to finding a Kauffman state $x^*$ with the maximum number of state circles.  
A recent result of Cohen, Kindred, Lowrance, Shanahan, and Van Cott specifies how to determine $\gamma(K)$ from any such optimal state $x^*$ \cite{cklsv}. See Theorem \ref{T:CKLSV}.

One way to find $x^*$ is to list and compare {\it all} the Kauffman states of $D$ (or just the {\it adequate} ones).  Alternatively, Adams and Kindred provide an algorithm for using 1-gons, bigons, and triangles in $S^2\setminus D$ to find $x^*$.  They remark on page 2969 that 
\begin{quotation}
 It would be straightforward to write a computer program to determine the [unoriented genus and crosscap number] of all knots [and links] in the census of alternating knots$\dots$
\end{quotation}
We do just this, using Adams and Kindred's minimal genus algorithm to compute the unoriented genus and crosscap number of every prime alternating link through 14 crossings and every prime alternating knot through 19 crossings.  

In order to automate the minimal genus algorithm, we develop a new (to our knowledge) method, given a Gauss code of a diagram $D$, of encoding the Kauffman states of $D$ as sets of tuples, which we call {\it state codes}.  Each tuple corresponds to a state circle, and each entry in a tuple identifies a crossing incident to the corresponding state circle.

A brief outline: \textsection\ref{S:Background} reviews the aforementioned results of Adams-Kindred and Cohen-Kindred-Lowrance-Shanahan-Van Cott along with Gauss codes, \textsection\ref{S:State codes} introduces state codes and describes how they develop under Adams-Kindred's minimal genus algorithm, and \textsection\ref{S:Results} presents the results of our computations.  The data and code associated with this project will be available at \href{https://github.com/IsaiasBahena/minimal-genus-state-surfaces}{this repository}.

\section{Background}\label{S:Background}

Assume throughout that $D\subset S^2$ is a $c$-crossing, connected, alternating diagram of a knot or link $L\subset S^3$, with $c>0$, and denote the components of $L$ by $L_1,\dots,L_n$.

The running assumption that $D$ is alternating will allow us to give particularly simple definitions of Gauss codes. We will similarly enjoy the liberty afforded by the fact that $L$ and its mirror image have the same unoriented genus and crosscap number, which will allow us to be unconcerned by the fact that our simplified Gauss codes will not distinguish among $D$, the mirror image of $D$, and the diagram obtained from $D$ by reversing all of its crossings (the last two diagrams both represent the mirror image of $L$).

\subsection{Gauss codes}\label{S:Gauss codes}

In the following way, a {\it Gauss code} encodes $D$ (up to reflection within or across the projection sphere) 
as a list $G=[G_1,\dots,G_n]$ of $n$ lists $G_i$, in which each of the integers $1,\dots,c$ appears exactly twice in total.  On each component $L_i$ of $L$, choose a starting point $p_i$ and orientation.  

First, record the list $G_1$ corresponding to $L_1$ as follows. Starting at $p_1$, trace $L_1$ through $D$, following its orientation. As you do, assign a label to each crossing in sequential order starting with 1. When a crossing is encountered for the second time, its label is repeated. Continue this process until the starting point is reached again. The Gauss code is the sequence of numbers written at each crossing in the order of traversal. For example, the diagram of the figure-8 knot $4_1$ shown (with choice of starting point and orientation) left in Figure \ref{Fi:Fig8} has Gauss code $[[1, 2, 3, 1, 4, 3, 2, 4]]$.

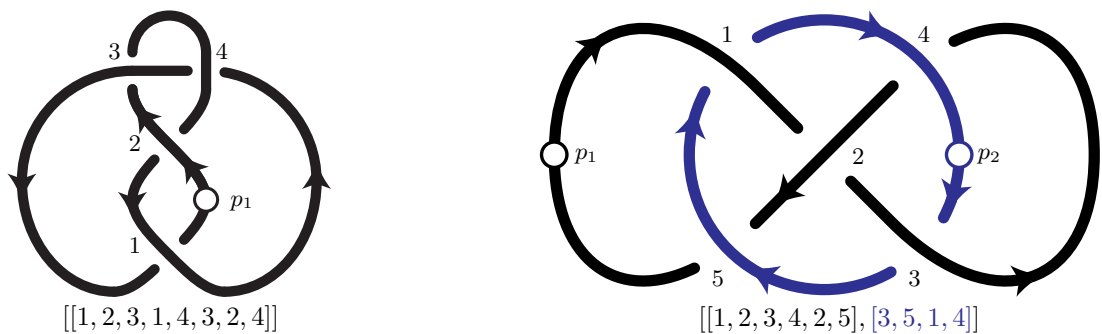
\begin{figure}[H]
    \begin{center}
    \labellist \hair 4pt
    \pinlabel{$[[1, 2, 3, 1, 4, 3, 2, 4]]$} at 160 -21
    \small
    \pinlabel{1} at 125 50
    \pinlabel{2} at 125 150
    \pinlabel{$p_1$} at 230 92
    \pinlabel{3} at 105 240
    \pinlabel{4} at 210 240
    \endlabellist
        \raisebox{9pt}{\includegraphics[height=1.5in]{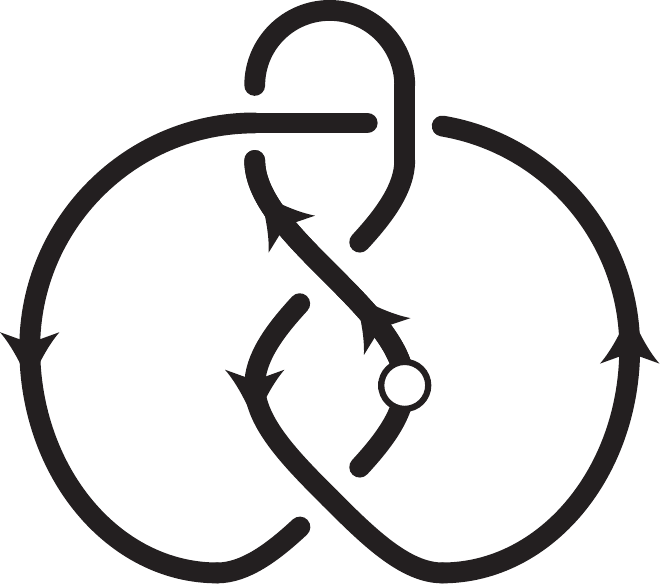}}
        \hspace{1in}
    \labellist  \hair 4pt
    \pinlabel{$ [[1, 2, 3, 4, 2, 5], {\color{Blue}[3, 5, 1, 4]}]$} at 160 -12
    \small
    \pinlabel{1} at 100 140
    \pinlabel{2} at 170 75
    \pinlabel{3} at 200 10
    \pinlabel{4} at 205 140
    \pinlabel{5} at 95 10
    \pinlabel{$p_1$} at 25 75
    \pinlabel{$p_2$} at 240 75
    \endlabellist
        \raisebox{9pt}{\includegraphics[height=1.5in]{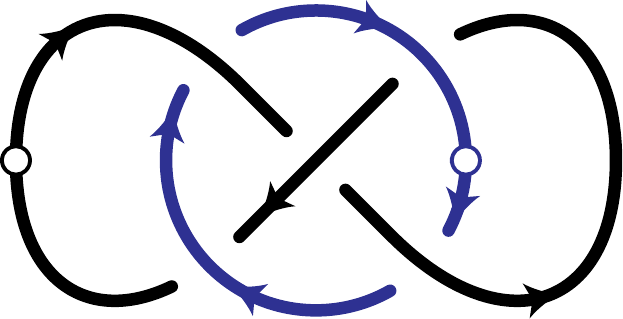}}
        \caption{Gauss codes for two oriented link diagrams with specified basepoints
         }
        \label{Fi:Fig8}
    \end{center}
\end{figure}
If $n>1$, repeat this process for $i=2,\dots,n$, in that order, maintaining all crossing labels as you go.  For example, the diagram of the Whitehead link $5^2_1$ shown (with choices of starting points and orientations) right in Figure \ref{Fi:Fig8} has Gauss code $ [[1, 2, 3, 4, 2, 5], [3, 5, 1, 4]]$.

There are typically many different Gauss codes for a given diagram $D$, depending on the ordering of the components $L_1,\dots,L_{n}$ of $L$ and on one's choice of starting point and orientation for each $L_i$.  Still, any two codes $G=[G_1,\dots,G_n]$ and $G''=[G''_1,\dots,G''_n]$ of $D$ are related by dihedrally permuting (i.e., possibly reversing and then cyclically permuting) each $G_i\to G'_i$, then arbitrarily permuting  $[G'_1,\dots,G'_n]\to[G'_{\sigma(1)},\dots,G'_{\sigma(n)}]$, and finally relabeling the crossings $1,\dots,c$ in $[G'_{\sigma(1)},\dots,G'_{\sigma(n)}]$ so that, whenever $1\leq i<j\leq c$, the first appearance of $i$ precedes that of $j$. In the dihedral permutation of $G_i$, reversal corresponds to reversing the orientation of $L_i$, and cyclic permutation corresponds to shifting the starting point $p_i$.

For example, in the diagram shown left in Figure \ref{Fi:Fig8}, shifting the starting point $p_1$ following the chosen orientation past one crossing changes its Gauss code from $[[1, 2, 3, 1, 4, 3, 2, 4]]$ to $[[2, 3, 1, 4, 3, 2, 4, 1]]$, which we relabel as $[[1,2,3,4,2,1,4,3]]$. Alternatively, keeping $p_1$ where it is in that diagram but reversing the orientation changes its Gauss code from $[[1, 2, 3, 1, 4, 3, 2, 4]]$ to $[[4,2,3,4,1,3,2,1]]$, which we relabel as $[[1,2,3,1,4,3,2,4]]$. The fact that this reversed Gauss code is identical to the original reflects the (fully amphichiral) symmetry of the figure-8 knot.

\subsection{Kauffman states and state surfaces}

There are two ways to resolve a given crossing in $D$. When $D$ is oriented, one of these ``smoothings'' respects the orientation on $D$, and we call it the {\it oriented smoothing}: $\raisebox{-.02in}{\includegraphics[width=.125in]{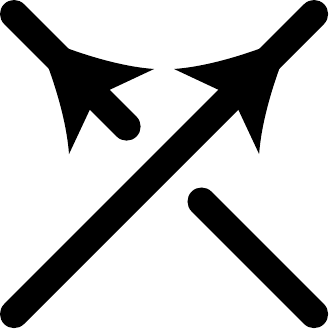}}
\longrightarrow\raisebox{-.02in}{\includegraphics[width=.125in]{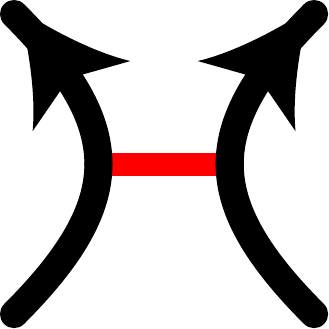}}$ (or the mirror image).  We call the other smoothing the {\it unoriented smoothing} of the crossing: $\raisebox{-.02in}{\includegraphics[width=.125in]{figures/CrossingOriented.pdf}}
\longrightarrow\raisebox{-.02in}{\includegraphics[width=.125in]{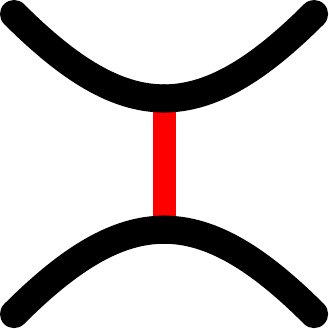}}$ (or the mirror image).

Smoothing all crossings in $D$ yields a {\it Kauffman state} $x$ of $D$, comprised of {\it state circles}.  Sometimes, we also include {\it crossing arcs} in a Kauffman state to keep track of where crossings were resolved in taking $D$ to $x$.  

Given a Kauffman state $x$ for $D$, one constructs its {\it state surface}  $F_x$ by capping off the state circles with mutually disjoint {\it state disks}, all on the same side of $S^2$, and attaching an appropriately half-twisted band at each crossing. See Figure \ref{Fi:State surface} for an example. %

\begin{figure}[H]
    \begin{center}
        \includegraphics[width=\textwidth]{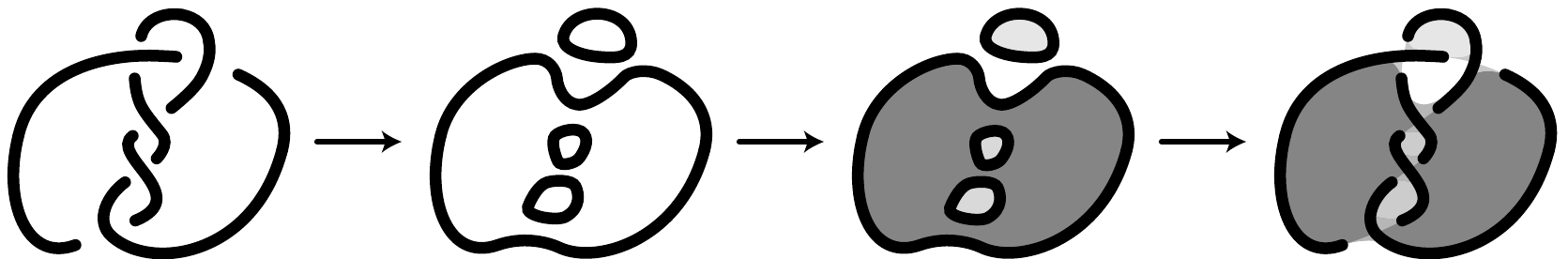}
        \caption{A diagram $D$, a Kauffman state $x$, and the state surface $F_x$}
        \label{Fi:State surface}
    \end{center}
\end{figure}

When $D$ has $c$ crossings and $x$ has $s$ state circles, the state surface $F_x$ admits a cell decomposition with $2c$ 0-cells (one on each overpass of $D$ and one on each underpass), $3c$ 1-cells (one ``vertical arc'' in each crossing band, and the rest running along the link), and $s$ 2-cells.  Therefore, $\chi(F_x)=s-c$.  Further, because $F_x$ is compact and connected with nonempty boundary, \[\beta_1(F_x)=1-\chi(F_x)=1+c-s.\]

Given a Kauffman state $x$ of $D$ (with crossing arcs), the associated {\it state graph} $G_x$ is constructed by collapsing each state circle to a point, so that $G_x$ has $s$ vertices and $c$ edges. In fact, $G_x$ can be viewed as a spine of $F_x$, and given a cycle $\gamma\subset G_x$, the number of crossing bands traversed by the corresponding cycle in $F_x$ will equal the edge-length of $\gamma$ in $G_x$.  One thus observes:

\begin{obs}
A state surface $F_x$ is 2-sided if and only if its state graph $G_x$ is bipartite.  Further, the state graph $G_x$ is simple if and only if:
\begin{itemize}
    \item Each crossing band joins {\it distinct} state disks (i.e., $x$ is {\it adequate}) and
    \item No two crossing bands join the same pair of state disks.
\end{itemize}
\end{obs}
    
\subsection{Unoriented genus and crosscap number of alternating links}

There is no known method to determine the unoriented genus or crosscap number of an arbitrary knot or link in $S^3$. In the alternating case, however, the following result provides just such a method.
\begin{theorem}[Adams and Kindred]\label{T:AK}
    Any nontrivial, connected, alternating diagram $D\subset S^2$ of a knot or link $L\subset S^3$ has Kauffman states $x,y$ whose state surfaces $F_x$ and $F_y$ realize the unoriented genus and crosscap numbers of $L$. 
\end{theorem}
This means that $\beta_1(F_x)=\Gamma(L)$, and $F_y$ is 1-sided with $\beta_1(F_y)=\gamma(L)$. Note that $x$ is a state of $D$ with the maximum number of state circles and therefore will always be adequate.

Thus, one could determine $\Gamma(L)$ and $\gamma(L)$ from $D$ by considering all of its states, choosing one, $x$, with the most state circles, and then, among the states whose state graphs are not bipartitite, choosing one, $y$, with the most state circles.  Better yet, one could just choose such an {\it optimal state} $x$, thus determining $\Gamma(L)$, and then apply the following result to determine $\gamma(L)$.

\begin{theorem}[Cohen, Kindred, Lowrance, Shanahan, and Van Cott \cite{cklsv}]\label{T:CKLSV}
    Suppose that $x$ is a state of a connected $c$-crossing prime alternating diagram $D\subset S^2$ of a link $L\subset S^3$, where $c\geq 3$, and let $G_x$ be its state graph. Then $\gamma(L)=\Gamma(L)+1$ if and only if $G_x$ is bipartite and simple. Otherwise, $\gamma(L)=\Gamma(L)$.
\end{theorem}

The hypotheses in Theorem \ref{T:CKLSV} that $D$ is prime and $c\geq 3$ are vital, as the result does not extend to the 2-crossing diagram of the Hopf link, nor to alternating diagrams that have that 2-crossing diagram as a connect summand. Later, when we compute crosscap numbers for all prime alternating links in the tables through 14 crossings, this will mean that we need to create a special condition in our code which ends up being specific to the Hopf link, as the Python program we develop to compute the remaining crosscap numbers will otherwise rely on Theorem \ref{T:CKLSV}. See \textsection\ref{S:Program} for more details.

One can find an optimal Kauffman state $x$ of $D$ by listing all of its states, or at least the adequate ones, and choosing one with the most state circles.  Alternatively, Adams and Kindred describe a procedure which streamlines the search. Before describing that procedure, we present a well-known property satisfied by the disks of $S^2\setminus D$.  We call such a disk an {\it $m$-gon} if its boundary contains $m$ crossings.

\begin{prop}\label{P:123gon}
    Any (connected) $c$-crossing knot or link diagram $D\subset S^2$ admits an $m$-gon for some $m\leq 3$.
\end{prop}

\begin{proof}
    Note first that $D$ induces a cell decomposition of $S^2$ with $c$ vertices, one at each crossing, and $2c$ edges, following the rest of $D$; hence, $D$ cuts $S^2$ into $\chi(S^2)-c+2c=c+2$ disks. For each $m>0$, let $f_m$ denote the number of $m$-gons.  Then $\sum_{m}f_m=c+2$, and $\sum_mmf_m=4c$ because each crossing abuts exactly four regions (counted with multiplicity).  Therefore,
\[8=4(c+2)-4c=4\sum_{m}f_m-\sum_{m}mf_m=\sum_m(4-m)f_m.\]
To equal 8, the sum on the right must have at least one positive summand, so at least one of $f_1$, $f_2$, or $f_3$ is positive.
\end{proof}

Before presenting the aforementioned procedure, we digress just slightly further:

\begin{prop}\label{P:c2}
    Every $c$-crossing prime alternating knot or link $L$, with $c\geq 3$,  satisfies $\Gamma(L)\leq c/2$ and $\gamma(L)\leq c/2 $.
\end{prop}

\begin{proof}
    Let $D$ be a $c$-crossing prime alternating diagram of $L$. Consider a checkerboard shading of $S^2\setminus D$, and let $F_1$ and $F_2$, respectively, be the checkerboard surfaces associated to the darkly and lightly shaded regions. Then $\beta_1(F_1)+\beta_1(F_2)=c$, so 
    \[\Gamma(L)\leq \min\{\beta_1(F_1),\beta_1(F_2)\}\leq c/2.\]
    Assume without loss of generality that $\beta_1(F_1)\leq \beta_1(F_2)$.  If any lightly shaded region of $S^2\setminus D$ is incident to an odd number of crossings, then $F_1$ is 1-sided, and so $\gamma(L)\leq \beta_1(F_1)\leq c/2$; the same conclusion holds if any lightly shaded region of $S^2\setminus D$ is a bigon, due to Theorem \ref{T:CKLSV}.  We may thus assume that every lightly shaded region of $S^2\setminus D$ is incident to at least four crossings (in fact, an even number of crossings), and thus to at least four edges.  Since each of the $2c$ edges of $D$ is incident to exactly one lightly shaded region, it follows that there are at most $2c/4$ lightly shaded regions, hence that $\beta_1(F_1)\leq c/2-1$.  Ergo, $\gamma(L)\leq \beta_1(F_1)+1\leq c/2$.
\end{proof}
    %
%

Here is the previously mentioned procedure of Adams and Kindred:  

\begin{min genus algo}[Adams and Kindred]
One can find a Kauffman state $x$ of $D$ with the most state circles as follows.  Its state surface $F_x$ will realize the unoriented genus of $L$.
\begin{enumerate}
    \item Identify the smallest $m$-gon (1-gon, 2-gon, or 3-gon).
    \item Smooth the crossings forming this $m$-gon to create a state circle. If $m=3$, create a second branch of the algorithm, where these three crossings are smoothed the opposite way.
    \item Repeat until no crossings remain. (If the diagram becomes disconnected, complete the procedure separately on each component.)
    \item From the list of resulting Kauffman states, choose one with the most state circles.
\end{enumerate}
\end{min genus algo}

\begin{figure}[H]
    \begin{center}
    \labellist
    \pinlabel{$p_1$} at 230 90
    \pinlabel{$3$} at 125 45
    \pinlabel{$1$} at 125 150
    \pinlabel{$2$} at 125 255
    \endlabellist
        \includegraphics[height=1.5in]{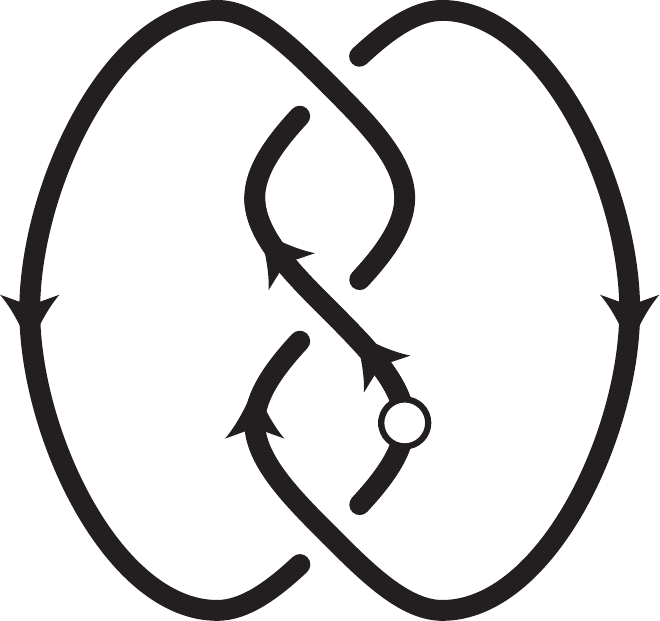}\raisebox{.5in}{$~\to~$}
            \labellist
    \pinlabel{$p_1$} at 230 90
    \pinlabel{{$3$}} at 125 45
    \pinlabel{{\color{red} $1$}} at 125 150
    \pinlabel{{\color{red} $2$}} at 125 255
    \endlabellist
    \includegraphics[height=1.5in]{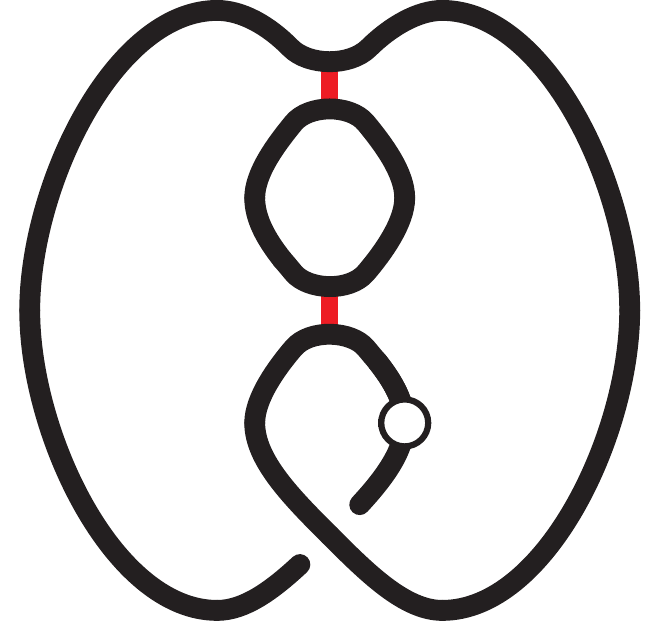}\raisebox{.5in}{$~\to~$}
        \labellist
    \pinlabel{$p_1$} at 230 90
    \pinlabel{{\color{red} $3$}} at 125 45
    \pinlabel{{\color{red} $1$}} at 125 150
    \pinlabel{{\color{red} $2$}} at 125 255
    \endlabellist
    \includegraphics[height=1.5in]{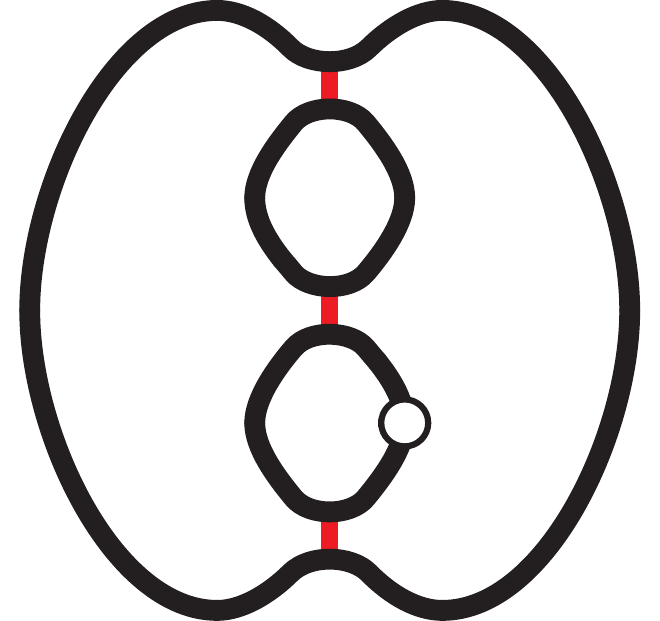}
        \caption{Applying the minimal genus algorithm to the trefoil}
        \label{Fi:minimal genus algorithm trefoil}
    \end{center}
\end{figure}
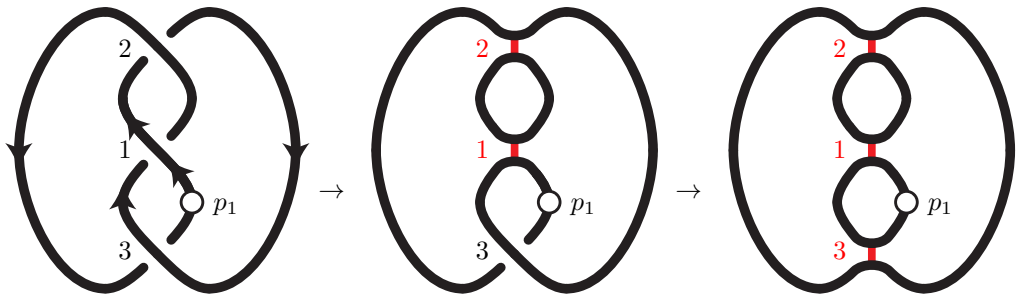
Figure \ref{Fi:minimal genus algorithm trefoil} shows the minimal genus algorithm being performed on the trefoil. 

\section{State codes}\label{S:State codes}

A {\it state code} is a list of tuples that encodes a Kauffman state, akin to how a Gauss code records a knot or link diagram.  Each tuple in the list describes a state circle, with each entry in the tuple identifying an incident crossing.  

Given a Gauss code $G$ for an alternating diagram $D$, we will describe how to modify $G$ while carrying out the minimal genus algorithm until we get an optimal state code (i.e., one with a maximal number of tuples, corresponding to a state with a maximal number of state circles).  At each step along the way (and for each branch of the procedure), we will have a hybrid {\it Gauss-state code}, which will be a list of three lists: (1) a ``near-Gauss'' code, (2) a list of tuples for state circles already formed, and (3) a list of which crossings have already been resolved.  A ``near-Gauss'' code may not quite be a Gauss code because some of its crossings may already have been resolved, but from now on we will simply call it a ``Gauss code'' nonetheless. Before we get into the details, we present a simple example.

\begin{example}\label{Ex:minimal genus algorithm trefoil}
    Consider the diagram $D$ of the trefoil shown left in Figure \ref{Fi:minimal genus algorithm trefoil}. Its Gauss code is $[[1, 2, 3, 1, 2, 3]]$. There are no 1-gons in $D$, but there are three bigons (one is the unbounded region) and two triangles. 
    
    Pick the bigon adjacent to crossings 1 and 2, and smooth crossings 1 and 2 so that this becomes a state circle, which we connect to the rest of the diagram via arcs to signify that the crossings are no longer there. The Gauss-state code at this stage, shown center in Figure \ref{Fi:minimal genus algorithm trefoil}, is 
    \[[[[1,3,2,3]],[(1,2)],[1,2]].\]
    Here, $[[1,3,2,3]]$ is the near-Gauss code (with the same choice of starting point and local orientation), $(1,2)$ is the state circle we have formed, and $[1,2]$ is the list of crossings we have smoothed.

    Now we consider the diagram obtained from this picture by deleting all state circles and crossing arcs.  (If it were disconnected, we would consider its components in isolation.) It has two 1-gons and a bigon--the unbounded region is a 1-gon, and the region that contained the state circle is a bigon because it abuts crossing 3 twice, and adjacent crossings are counted with multiplicity. Smooth the 1-gon adjacent to crossing arc 1 and crossing 3 (note that smoothed crossings do not factor into identifying an $m$-gon). 
    The Gauss-state code at this stage, shown right in Figure \ref{Fi:minimal genus algorithm trefoil}, would be 
    \[[[[3,2]],[(1,2),(1,3)],[1,2,3]],\] but all of the crossings have been resolved and so we can re-designate $[3,2]$ as a state circle, giving us the {\it state code} $[(1,2),(1,3),(3,2)]$ for this Kauffman state.    
\end{example}

Notice that in Example \ref{Ex:minimal genus algorithm trefoil} we got the Kauffman state with three state circles as opposed to one with two state circles or fewer. Indeed, the minimal genus algorithm guarantees that its output is an optimal Kauffman state, one with the most state circles.

Now, how do we use Gauss codes to implement the minimal genus algorithm without relying on visual diagrams? 

To answer this question, the remainder of this section details how each type of $m$-gon is detected and smoothed using only Gauss codes (and near-Gauss codes). We begin by explaining how to identify and smooth 1-gons, which are the simplest to detect. We then move onto bigons, which require us to handle four cases, depending on how the two strands of the bigon are oriented and whether or not they come from the same component of the link. Finally, we address triangles, which are the most complex and introduce the need for branching via ``triangle'' and ``anti-triangle'' smoothings. 

\subsection{Detecting and smoothing a 1-gon in a Gauss code}\label{S:1gon}
A \textbf{1-gon} is present in a Gauss code $G=[G_1,\dots,G_n]$ when a crossing is repeated in the cyclic reading of some component of the Gauss code with only smoothed crossings between its two appearances. We sometimes say that the two appearances of that crossing are ``effectively consecutive.'' For example, the Gauss code $[[1,3,2,3]]$ in the Gauss-state code $[[[1,3,2,3]],[(1,2)],[1,2]]$ from the middle stage in Example \ref{Ex:minimal genus algorithm trefoil} has two 1-gons, namely $3,2,3$ and $3,1,3$.  

Whenever we identify a 1-gon, we 
permute the component of the Gauss code in which the 1-gon appears so that it has the form $[\boldsymbol{a, s_0, a}, w_0]$, where \( \boldsymbol{a} \) is the crossing, \( \boldsymbol{s_0}\) consists only of smoothed crossings (and is possibly empty), and $w_0$ is an arbitrary part of the code. For the purposes of this paper, we find no need to relabel crossings after we perform these permutations; likewise for the rest of this section. 

\begin{figure}[H]
    \begin{center}
        \labellist
        \pinlabel{$\boldsymbol{a}$} at 265 160
        \pinlabel{$\boldsymbol{s_0}$} at 75 135
        \pinlabel{$p_1$} at 280 90
        \pinlabel{$w_0$} at 390 130
        \endlabellist
        \includegraphics[width=.4\textwidth]{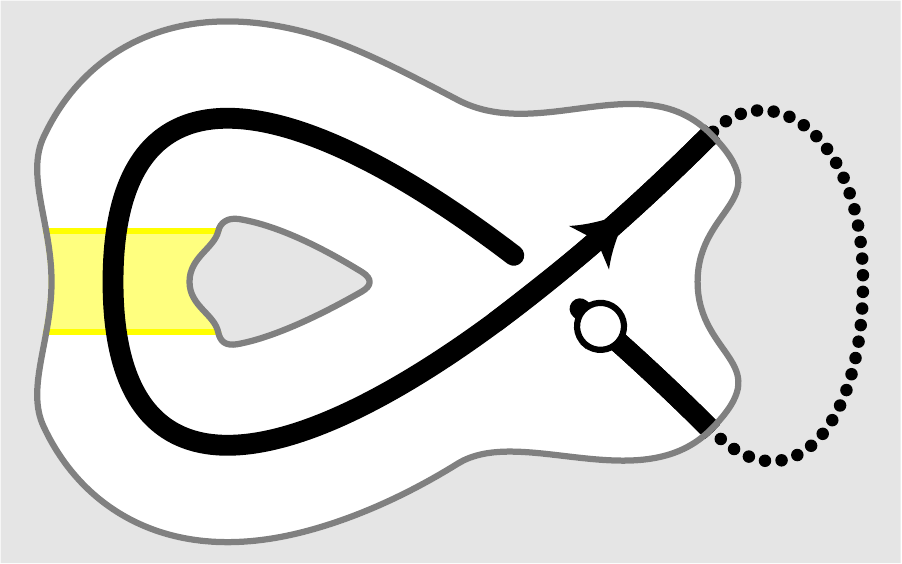}\raisebox{.75in}{$~\to~$}
                \labellist
        \pinlabel{{\color{red} $\boldsymbol{a}$}} at 260 160
        \pinlabel{$\boldsymbol{s_0}$} at 75 135
        \pinlabel{$p_1$} at 280 90
        \pinlabel{$w_0$} at 390 130
        \endlabellist
        \includegraphics[width=.4\textwidth]{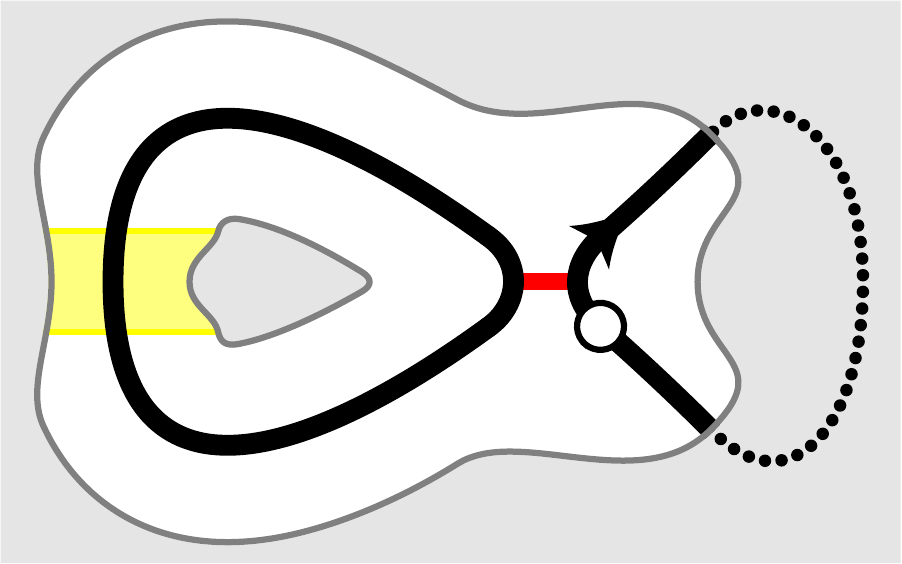}
        \caption{A 1-gon before (left) and after (right) we smooth it}
        \label{Fi:1Gon}
    \end{center}
\end{figure}
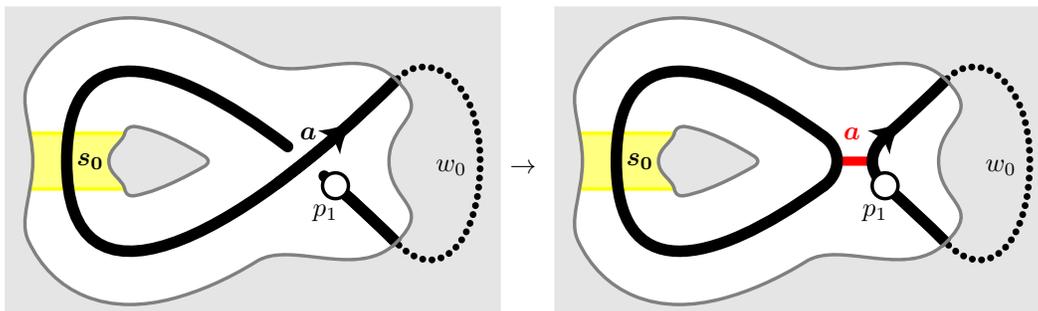

Once we detect a 1-gon and permute its component of the Gauss code to be $[\boldsymbol{a, s_0, a}, w_0]$, we smooth the crossing to convert the boundary of the 1-gon into a state circle, as shown in Figure \ref{Fi:1Gon}. The new state circle is $(\boldsymbol{a,s_0})$, the new Gauss code is $[[\boldsymbol{a},w_0],G_{2},\dots,G_n]$, and $\boldsymbol{a}$ is added to the list of smoothed crossings. A 1-gon smoothing always preserves orientation.

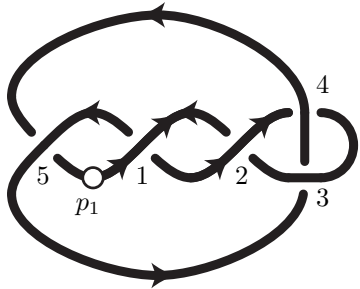
\begin{figure}[H]
    \begin{center}
        \labellist
        \pinlabel{$p_1$} at 90 90
        \pinlabel{5} at 40 125
        \pinlabel{1} at 150 125
        \pinlabel{2} at 260 125
        \pinlabel{3} at 350 100
        \pinlabel{4} at 350 225
        \endlabellist
        \includegraphics[height=1.5in]{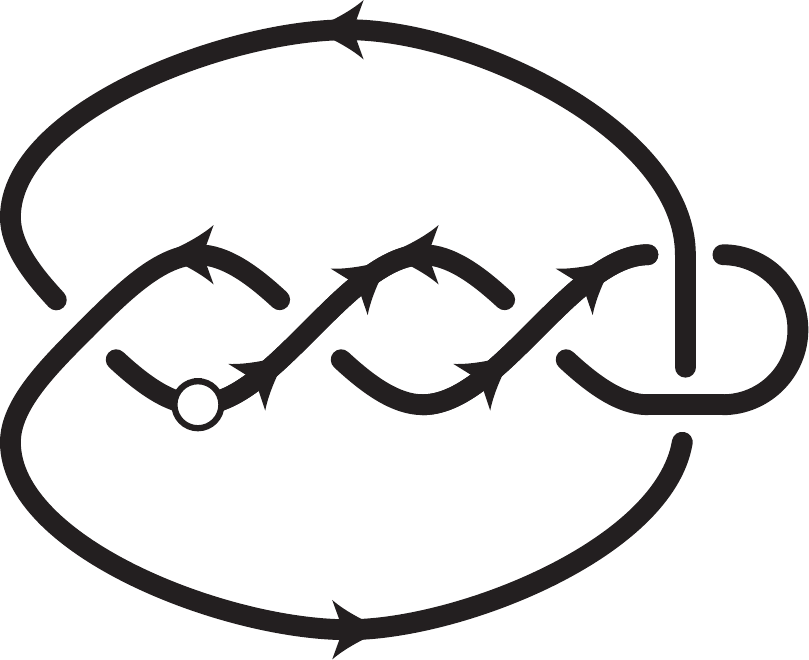}
        \caption{This diagram of the $5_2$ knot has three bigons.}
        \label{Fi:5_2}
    \end{center}
\end{figure}

\subsection{Detecting and smoothing a bigon in a Gauss code}\label{S:Bigon}
A bigon is present in a Gauss code when two distinct crossings appear effectively consecutively twice, meaning that only smoothed crossings appear in between them each time. For example, if we have a Gauss code of $[[1, 2, 3, 1, 4, 5, 2, 3, 5, 4]]$, as in Figure \ref{Fi:5_2}, then there are three bigons present. They are $(2, 3)$, $(4, 5)$, and $(1, 4)$. Notice that the second occurrence of $(4, 5)$ appears as $(5, 4)$, which is okay. Also note that the second occurrence of $(1, 4)$ wraps around the Gauss code, which is also okay because of the cyclic nature of Gauss codes. 

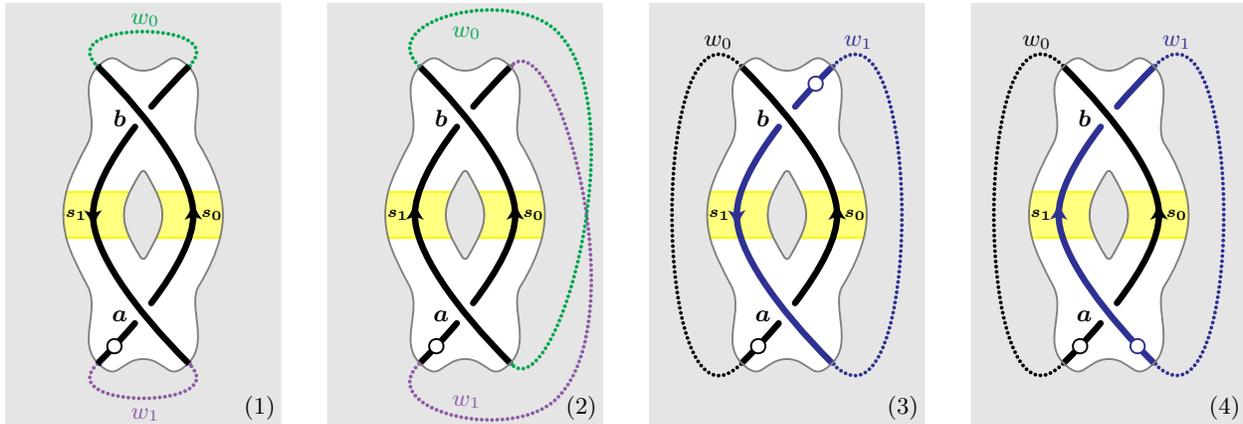
\begin{figure}[H]
    \begin{center}
        \labellist \small \hair 4pt
        \pinlabel{$\boldsymbol{a}$} at 180 175
        \pinlabel{$\boldsymbol{b}$} at 180 485
        \pinlabel{{\color{Green} $w_0$}} at 220 640
        \pinlabel{{\color{Purple} $w_1$}} at 220 20
        \pinlabel{{(1)}} at 400 30
        \tiny
        \pinlabel{$\boldsymbol{s_0}$} at 325 333
        \pinlabel{$\boldsymbol{s_1}$} at 112 333
        \small
        \pinlabel{$\boldsymbol{a}$} at 685 175
        \pinlabel{$\boldsymbol{b}$} at 685 485
        \pinlabel{{\color{Green} $w_0$}} at 725 620
        \pinlabel{{\color{Purple} $w_1$}} at 725 40
        \pinlabel{{(2)}} at 905 30
        \pinlabel{{(3)}} at 1410 30
        \tiny
        \pinlabel{$\boldsymbol{s_0}$} at 830 333
        \pinlabel{$\boldsymbol{s_1}$} at 617 333
        \small
        \pinlabel{$\boldsymbol{a}$} at 1190 175
        \pinlabel{$\boldsymbol{b}$} at 1190 485
        \pinlabel{{\color{black} $w_0$}} at 1120 605
        \pinlabel{{\color{Blue} $w_1$}} at 1340 605
        \pinlabel{{(4)}} at 1915 30
        \tiny
        \pinlabel{$\boldsymbol{s_0}$} at 1335 333
        \pinlabel{$\boldsymbol{s_1}$} at 1122 333
        \small
        \pinlabel{$\boldsymbol{a}$} at 1695 175
        \pinlabel{$\boldsymbol{b}$} at 1695 485
        \pinlabel{{\color{black} $w_0$}} at 1620 605
        \pinlabel{{\color{Blue} $w_1$}} at 1840 605
        \tiny
        \pinlabel{$\boldsymbol{s_0}$} at 1840 333
        \pinlabel{$\boldsymbol{s_1}$} at 1627 333
        \endlabellist   
        \includegraphics[width=\textwidth]{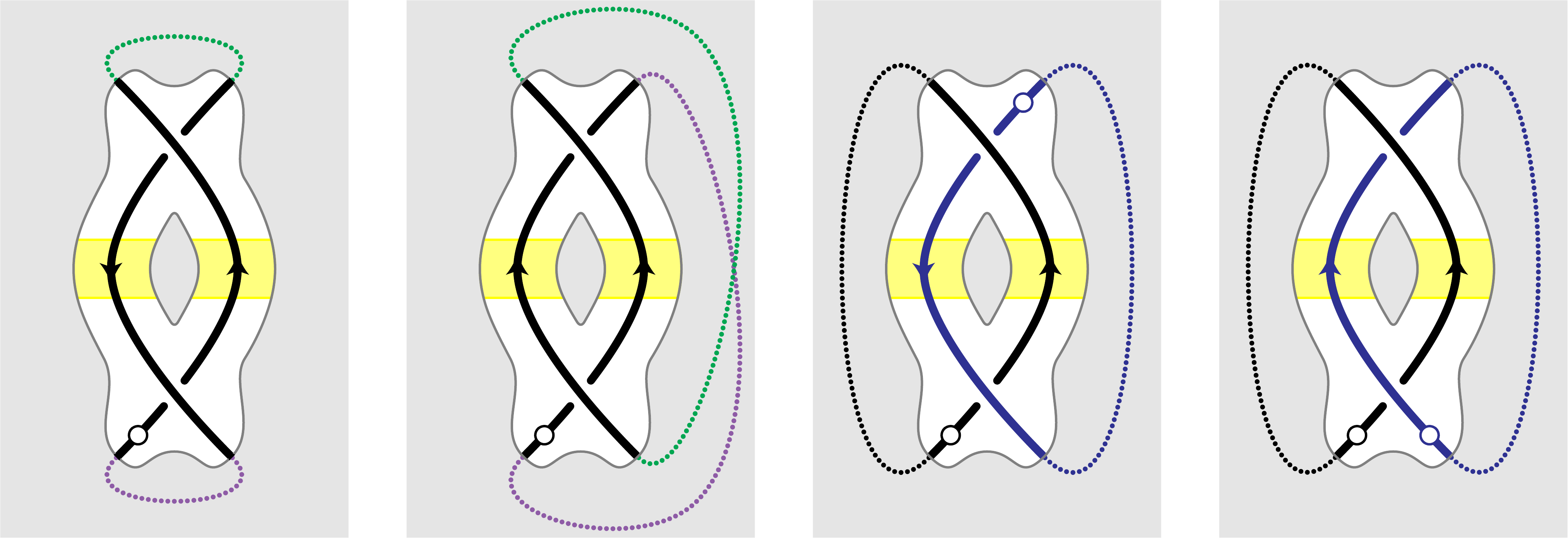}
        \caption{ The four different cases of bigons}
        \label{Fi:Detecting 2-gon}
    \end{center}
\end{figure}

Whenever we identify a bigon, we permute the components $G_i$ of the Gauss code so that the crossings $\boldsymbol{a}$ and $\boldsymbol{b}$ in the bigon appear in $G_1$, and possibly $G_2$, and then we permute $G_1$, and possibly $G_2$, so that their first two unsmoothed crossings are $\boldsymbol{a}$ and $\boldsymbol{b}$.

As shown in Figure \ref{Fi:Detecting 2-gon}, there are four different types of Gauss codes that these permutations may produce, depending on how the two strands of the bigon are oriented and whether or not they come from the same component of the link. After permuting the components $G_i$ of the Gauss code so that the 1-gon appears in $G_1$ and possibly $G_2$ and then permuting $G_1$ and possibly $G_2$, the four different formats for a bigon (following the four cases in Figure \ref{Fi:Detecting 2-gon} from left to right) are:

\begin{enumerate}
    \item oriented knot: $[[\boldsymbol{a, s_0, b}, w_0, \boldsymbol{b, s_1, a}, w_1],G_2,\dots,G_n]$,
    \item unoriented knot: $[[\boldsymbol{a, s_0, b}, w_0, \boldsymbol{a, s_1, b}, w_1],G_2,\dots,G_n]$,
    \item oriented link: $[[\boldsymbol{a, s_0, b}, w_0],[ \boldsymbol{b, s_1, a}, w_1],G_3,\dots,G_n]$, and
    \item unoriented link: $[[\boldsymbol{a, s_0, b}, w_0],[ \boldsymbol{a, s_1, b}, w_1],G_3,\dots,G_n]$.
\end{enumerate}

Here, $\boldsymbol{a}$ and $\boldsymbol{b}$ are crossings, $\boldsymbol{s_0}$ and $\boldsymbol{s_1}$ are segments of (components of) the Gauss code which contain only smoothed crossings, and $w_0$ and $w_1$ are arbitrary segments of the Gauss code.  The descriptor ``oriented'' means that the two strands of the bigon are oriented opposite each other, so the smoothings we perform are oriented; ``unoriented'' is the alternative.  Finally, ``knot'' here means that the two strands of the bigon come from the same link component, and ``link'' means the opposite. Note that every bigon from a Gauss code for a knot is in one of the first two cases, although the converse is false.

\begin{figure}[H]
    \begin{center}
        \labellist \small \hair 4pt
        \pinlabel{$\boldsymbol{a}$} at 180 175
        \pinlabel{$\boldsymbol{b}$} at 180 485
        \pinlabel{{\color{Green} $w_0$}} at 220 640
        \pinlabel{{\color{Purple} $w_1$}} at 220 20
        \tiny
        \pinlabel{$\boldsymbol{s_0}$} at 325 333
        \pinlabel{$\boldsymbol{s_1}$} at 112 333
        \small
        \pinlabel{{\color{red} $\boldsymbol{a}$}} at 685 175
        \pinlabel{{\color{red} $\boldsymbol{b}$}} at 685 485
        \pinlabel{{\color{Green} $w_0$}} at 725 640
        \pinlabel{{\color{Purple} $w_1$}} at 725 20
        \tiny
        \pinlabel{$\boldsymbol{s_0}$} at 830 333
        \pinlabel{$\boldsymbol{s_1}$} at 617 333
        \small
        \pinlabel{$\boldsymbol{a}$} at 1190 175
        \pinlabel{$\boldsymbol{b}$} at 1190 485
        \pinlabel{{\color{Green} $w_0$}} at 1230 620
        \pinlabel{{\color{Purple} $w_1$}} at 1230 40
        \pinlabel{(1)} at 470 300
        \pinlabel{(2)} at 1480 300
        \tiny
        \pinlabel{$\boldsymbol{s_0}$} at 1335 333
        \pinlabel{$\boldsymbol{s_1}$} at 1122 333
        \small
        \pinlabel{{\color{red} $\boldsymbol{a}$}} at 1695 175
        \pinlabel{{\color{red} $\boldsymbol{b}$}} at 1695 485
        \pinlabel{{\color{Green} $w_0$}} at 1735 620
        \pinlabel{{\color{Purple} $w_1$}} at 1735 40
        \tiny
        \pinlabel{$\boldsymbol{s_0}$} at 1840 333
        \pinlabel{$\boldsymbol{s_1}$} at 1627 333
        \endlabellist
        \includegraphics[width=\textwidth]{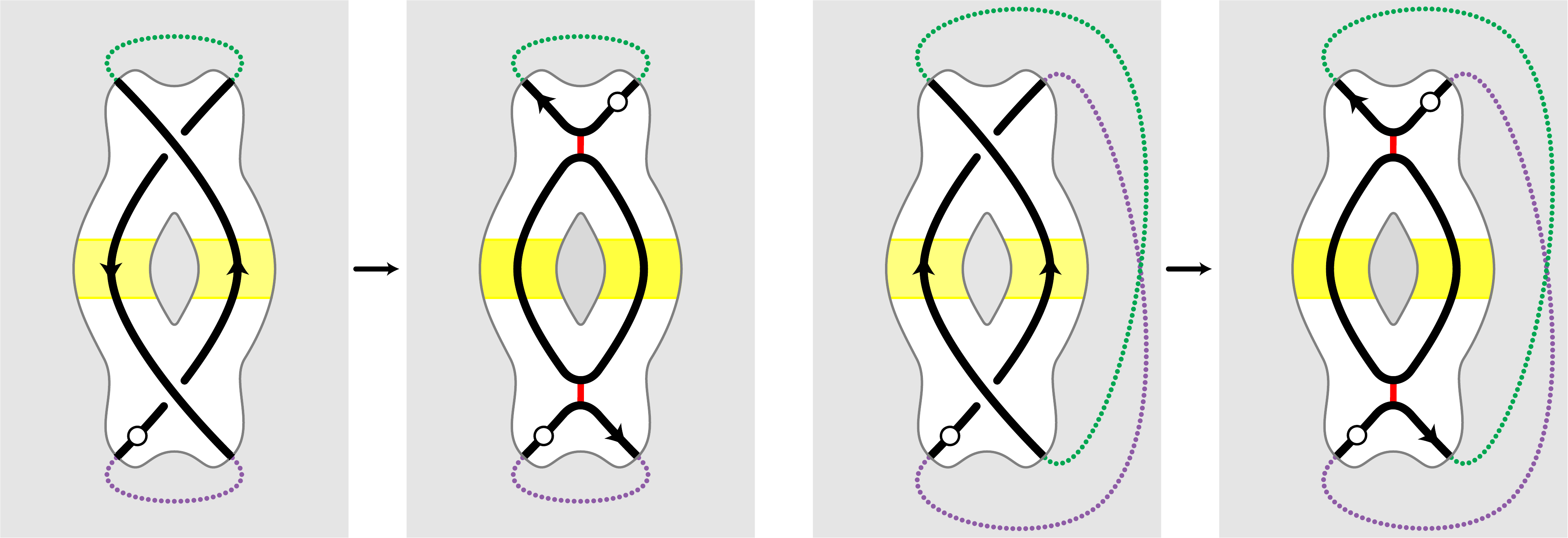}
        \caption{Smoothing a bigon formed by two strands of the same component of a knot or link}
        \label{Fi:Smoothing Knot bigon 1}
    \end{center}
\end{figure}
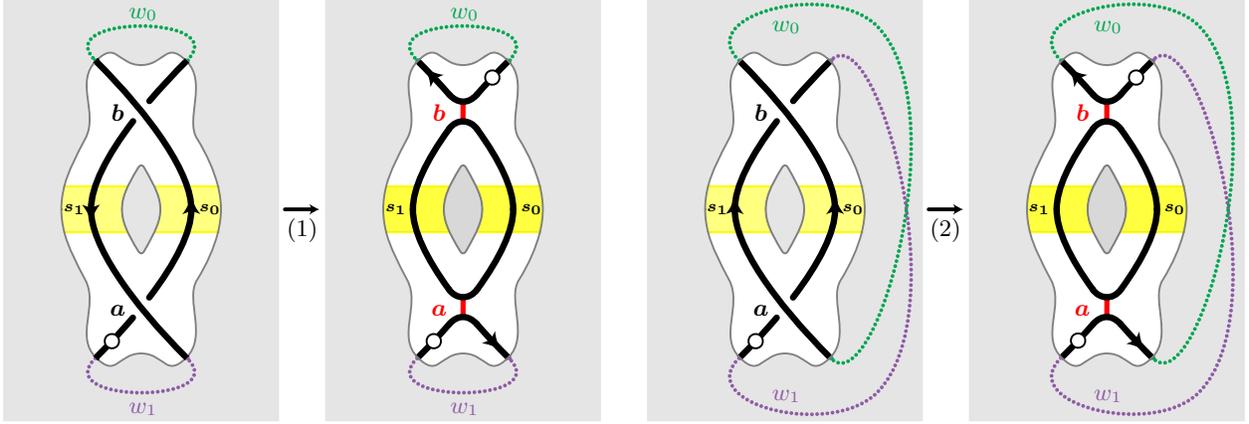

Once we detect a bigon and permute its component(s) of the Gauss code, we smooth the crossings so that the boundary of the bigon becomes a state circle.  We add $\boldsymbol{a}$ and $\boldsymbol{b}$ to the list of smoothed crossings. In the oriented cases, we add $(\boldsymbol{a,s_0,b,s_1})$ to the list of state circles, whereas in the unoriented cases we add $(\boldsymbol{a,s_0,b,\overline{s_1}})$. (Here, as in the sequel, the bar in $\mathbf{\overline{s_1}}$ indicates that the sequence $\boldsymbol{s_1}$ is reversed.) In each case, the Gauss code changes as follows (see Figures \ref{Fi:Smoothing Knot bigon 1} and \ref{Fi:Smoothing Knot bigon 2}).
\begin{align*}
(1)~&[[\boldsymbol{a, s_0, b}, w_0, \boldsymbol{b, s_1, a}, w_1],G_2,\dots,G_n]\to[[\boldsymbol{a},w_1], [\boldsymbol{b}, w_0],G_2,\dots,G_n],\\
(2)~&[[\boldsymbol{a, s_0, b}, w_0, \boldsymbol{a, s_1, b}, w_1],G_2,\dots,G_n]\to[[\boldsymbol{a},\overline{w_0},\boldsymbol{b}, w_1],G_2,\dots,G_n],\\
(3)~&[[\boldsymbol{a, s_0, b}, w_0],[ \boldsymbol{b, s_1, a}, w_1],G_3,\dots,G_n]\to[[\boldsymbol{a},w_1,\boldsymbol{b}, w_0],G_3,\dots,G_n] \text{, and}\\
(4)~&[[\boldsymbol{a, s_0, b}, w_0],[ \boldsymbol{a, s_1, b}, w_1],G_3,\dots,G_n]\to[[\boldsymbol{a},\overline{w_1},\boldsymbol{b}, w_0],G_3,\dots,G_n].
\end{align*}

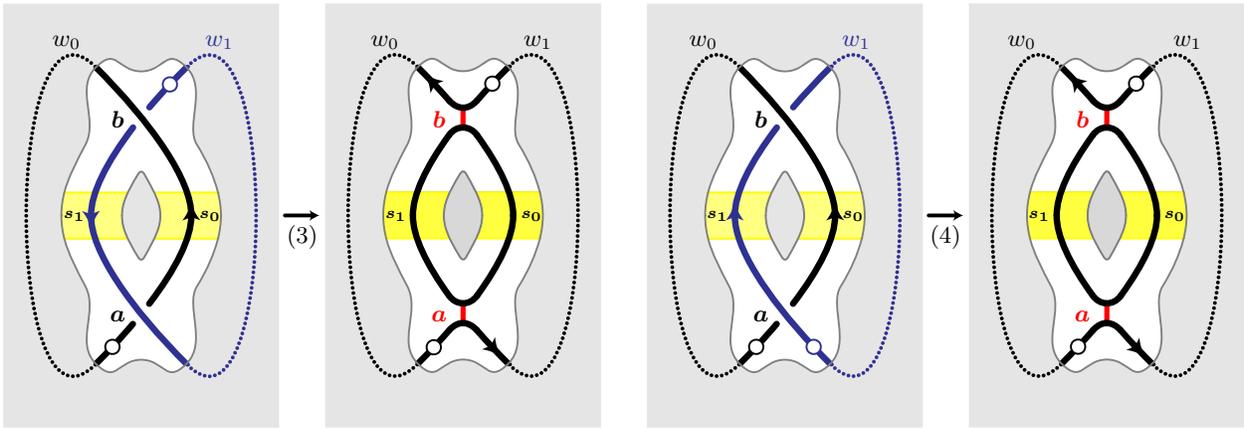
\begin{figure}[H]
    \begin{center}
        \labellist \small \hair 4pt
        \pinlabel{$\boldsymbol{a}$} at 180 175
        \pinlabel{$\boldsymbol{b}$} at 180 485
        \pinlabel{{\color{black} $w_0$}} at 100 605
        \pinlabel{{\color{Blue} $w_1$}} at 340 605
        \tiny
        \pinlabel{$\boldsymbol{s_0}$} at 325 333
        \pinlabel{$\boldsymbol{s_1}$} at 112 333
        \small
        \pinlabel{{\color{red} $\boldsymbol{a}$}} at 685 175
        \pinlabel{{\color{red} $\boldsymbol{b}$}} at 685 485
        \pinlabel{{\color{black} $w_0$}} at 600 605
        \pinlabel{{\color{black} $w_1$}} at 840 605
        \tiny
        \pinlabel{$\boldsymbol{s_0}$} at 830 333
        \pinlabel{$\boldsymbol{s_1}$} at 617 333
        \small
        \pinlabel{$\boldsymbol{a}$} at 1190 175
        \pinlabel{$\boldsymbol{b}$} at 1190 485
        \pinlabel{{\color{black} $w_0$}} at 1100 605
        \pinlabel{{\color{Blue} $w_1$}} at 1340 605
        \pinlabel{(3)} at 470 300
        \pinlabel{(4)} at 1480 300
        \tiny
        \pinlabel{$\boldsymbol{s_0}$} at 1335 333
        \pinlabel{$\boldsymbol{s_1}$} at 1122 333
        \small
        \pinlabel{{\color{red} $\boldsymbol{a}$}} at 1695 175
        \pinlabel{{\color{red} $\boldsymbol{b}$}} at 1695 485
        \pinlabel{{\color{black} $w_0$}} at 1600 605
        \pinlabel{{\color{black} $w_1$}} at 1860 605
        \tiny
        \pinlabel{$\boldsymbol{s_0}$} at 1840 333
        \pinlabel{$\boldsymbol{s_1}$} at 1627 333
        \endlabellist
        \includegraphics[width=\textwidth]{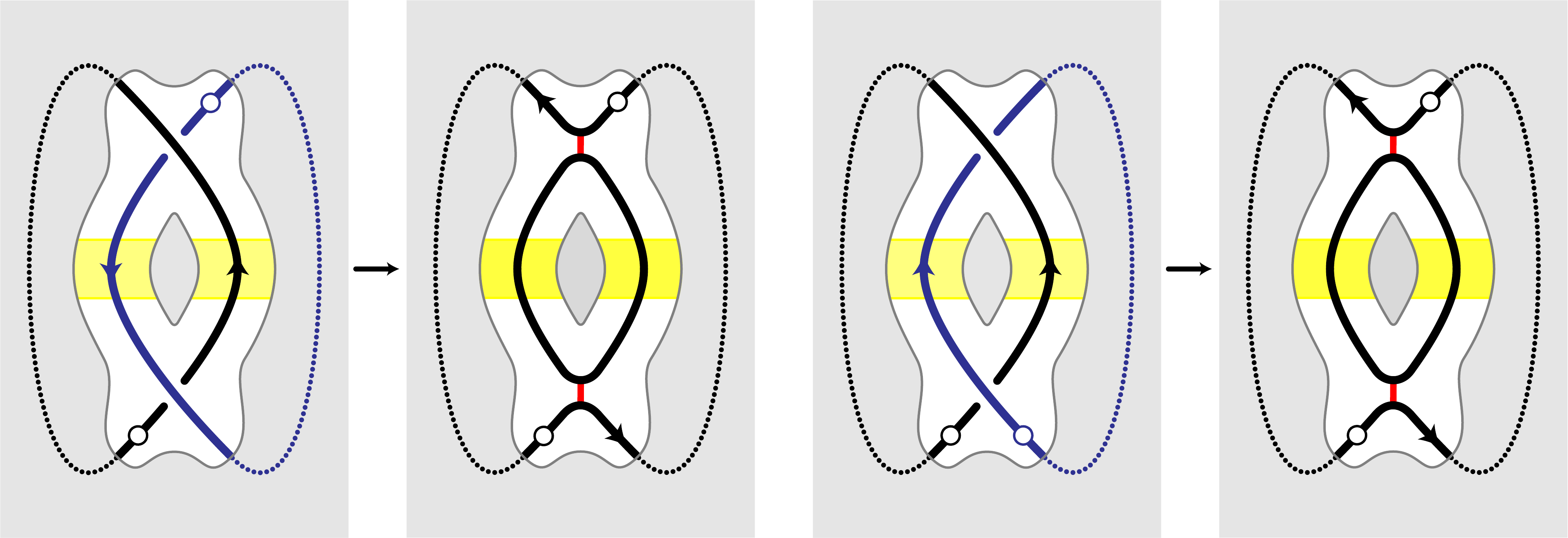}
        \caption{Smoothing a bigon formed by strands of distinct components of a link}
        \label{Fi:Smoothing Knot bigon 2}
    \end{center}
\end{figure}

Note that smoothing a bigon can change the number of components in a Gauss code. 
 In case (1), this number increases by one, whereas in cases (3) and (4) it decreases by one.

\subsection{Detecting and smoothing a triangle in a Gauss code} \label{S:Detect}
A triangle is present in a Gauss code when three crossings $\boldsymbol{a}$, $\boldsymbol{b}$, and $\boldsymbol{c}$ appear effectively consecutively in each of their three (unordered) pairs, $\{\boldsymbol{a, b}\}$, $\{\boldsymbol{b, c}\}$, and $\{\boldsymbol{a, c}\}$. Because there are now three crossings and three strands involved, there are many different ways that everything can be connected, and there are many different formats for what a triangle could look like in a Gauss code. Before getting into the details of the possible formats, we illustrate the main ideas through a concrete example, which will be a running example throughout \textsection\ref{S:Detect}.

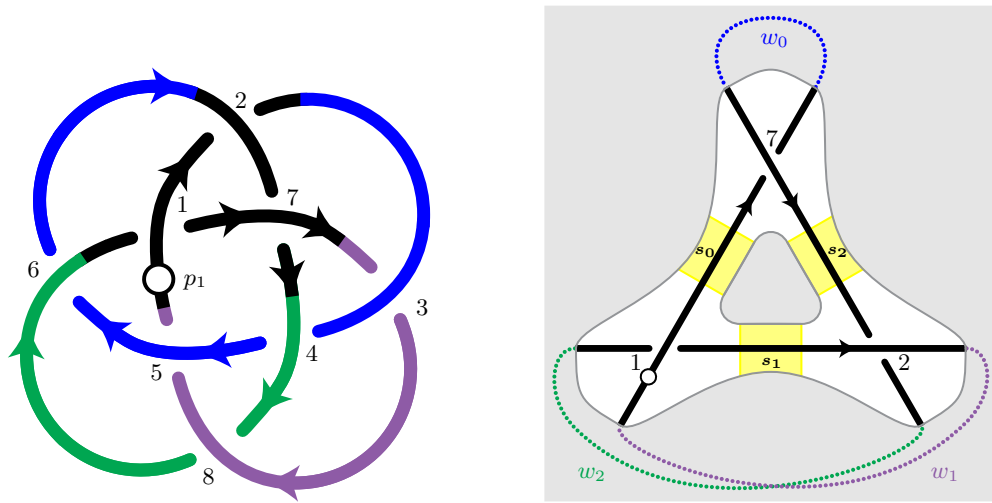
\begin{figure}[H]
    \begin{center}
        \labellist
        \small \hair 4pt
        \pinlabel{$p_1$} at 145 170
        \pinlabel{1} at 135 230
        \pinlabel{2} at 180 310
        \pinlabel{3} at 320 150
        \pinlabel{4} at 235 115
        \pinlabel{5} at 115 100
        \pinlabel{6} at 20 180
        \pinlabel{7} at 220 235
        \pinlabel{8} at 155 20
        \endlabellist
        \includegraphics[height=.35\textwidth]{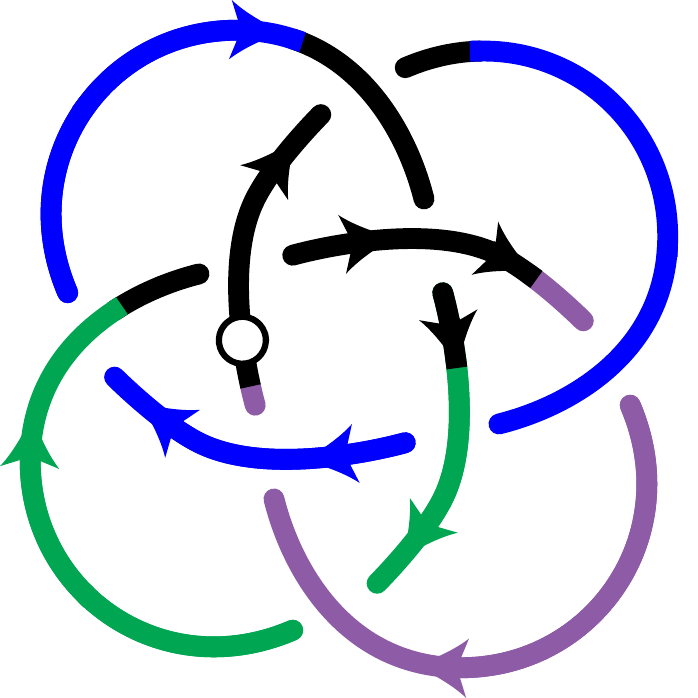}\hspace{.075\textwidth}
        \labellist
        \small \hair 4pt
        \pinlabel{1} at 135 205
        \pinlabel{2} at 530 205
        \pinlabel{7} at 335 535
        \pinlabel{{\color{blue} $w_0$ }} at 345 680
        \pinlabel{{\color{Purple} $w_1$}} at 590 40
        \pinlabel{{\color{Green} $w_2$}} at 70 40
        \tiny
        \pinlabel{$\boldsymbol{s_0}$} at 237 368
        \pinlabel{$\boldsymbol{s_1}$} at 335 205
        \pinlabel{$\boldsymbol{s_2}$} at 432 368
        \endlabellist
        \includegraphics[height=.4\textwidth]{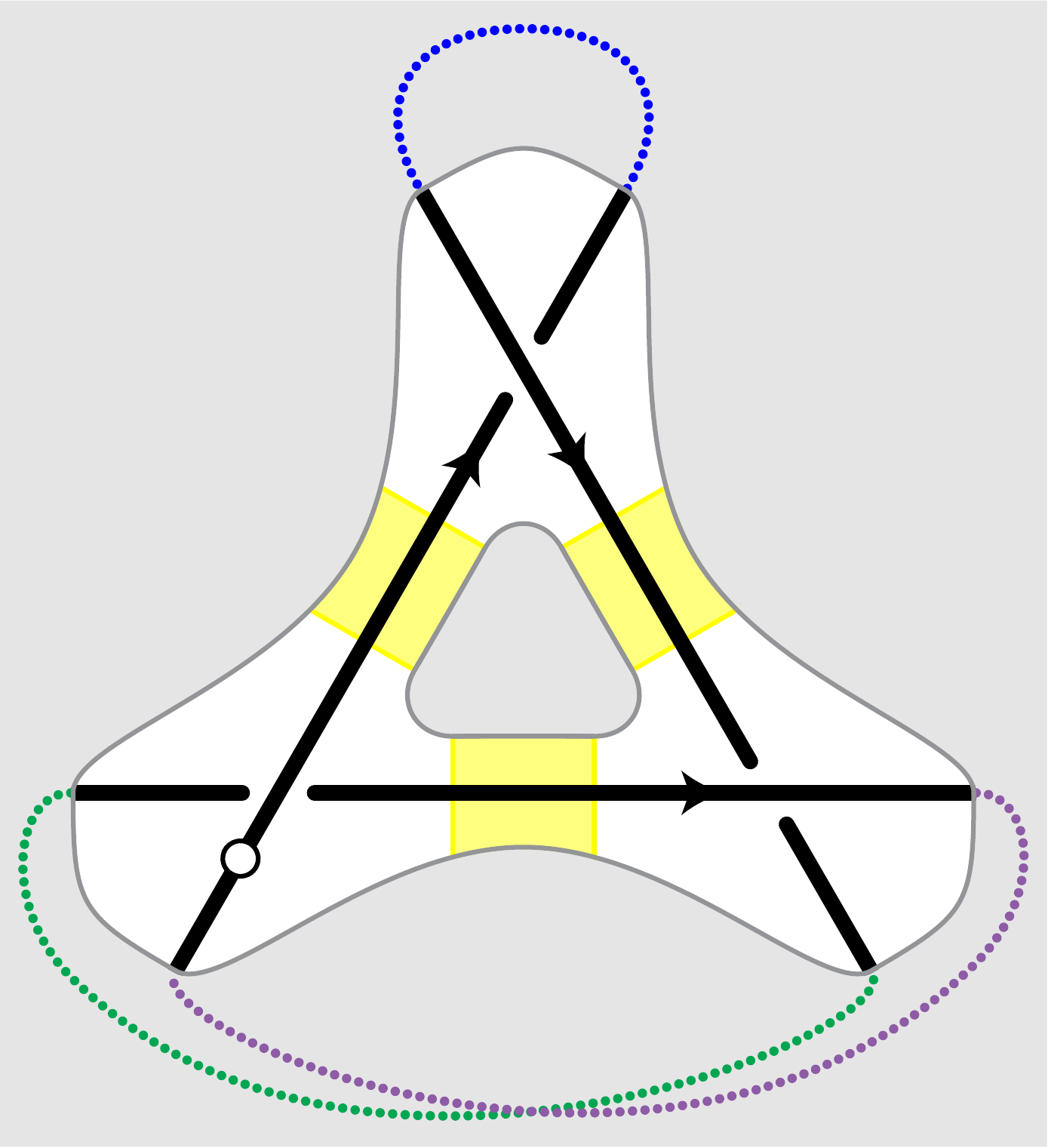}
        \caption{A diagram of the knot $8_{18}$ with Gauss code  $[[1, 2, 3, 4, 5, 6, 2, 7, 4, 8, 6, 1, 7, 3, 8, 5]]$ and a triangle formed by crossings 1, 2, and 7. Also see \eqref{E:Triangle1}, \eqref{E:triangle2},  \eqref{E:Triangle3}, and Figures \ref{Fi:8_18 Triangle Smoothing}-\ref{Fi:8_18 Anti-Triangle Smoothing}.}
        \label{Fi:8_18}
    \end{center}
\end{figure}

\begin{example}\label{Ex:8_18}
    Consider the diagram of the $8_{18}$ knot shown left in Figure \ref{Fi:8_18}. Its Gauss code is $G=[[\boldsymbol{1}, \boldsymbol{2}, 3, 4, 5, 6, \boldsymbol{2, 7}, 4, 8, 6, \boldsymbol{1, 7}, 3, 8, 5]]$. Thus, crossings $\boldsymbol{1}$, $\boldsymbol{2}$, and $\boldsymbol{7}$ form a triangle, because we see $\boldsymbol{1,2}$ and $\boldsymbol{2,7}$ and $\boldsymbol{1,7}$ consecutively in $G$.  
Figure \ref{Fi:8_18}, right, zooms in on the triangle itself and connects the three arbitrary strands based on how we know they connect in $8_{18}$.  The format here is 
\begin{equation}\label{E:Triangle1}
    G=[[\boldsymbol{a}, \boldsymbol{s_0}, \boldsymbol{b}, w_0, \boldsymbol{b}, \boldsymbol{s_1}, \boldsymbol{c}, w_1, \boldsymbol{a}, \boldsymbol{s_2}, \boldsymbol{c}, w_2]],
\end{equation}
where, as in \textsection\textsection\ref{S:1gon}-\ref{S:Bigon}, each $\boldsymbol{s_i}$ is a segment of $G$ that includes only smoothed crossings and each $w_i$ is an arbitrary segment or word within $G$. In this particular example, $\boldsymbol{a}=\boldsymbol{1}$, $\boldsymbol{b}=\boldsymbol{2}$, $\boldsymbol{c}=7$, each $\boldsymbol{s_i}$ is empty, $w_0=[3,4,5,6]$, $w_1=[4,8,6]$, and $w_2=[3,8,5]$. 
\end{example}

In other examples, however, the crossings could appear in different orders or in different components. Still, up to permutation, there are three distinct {\it types} of formats:
\begin{equation}\label{E:triangle2}
\begin{split}
    G&=[[\boldsymbol{x_{00}}, \boldsymbol{s_0}, \boldsymbol{x_{01}}, w_0, \boldsymbol{x_{10}}, \boldsymbol{s_1},\boldsymbol{x_{11}}, w_1, \boldsymbol{x_{20}}, \boldsymbol{s_2}, \boldsymbol{x_{21}}, w_2],G_2,\dots,G_n],\\
G&=[[\boldsymbol{x_{00}}, \boldsymbol{s_0}, \boldsymbol{x_{01}}, w_0, \boldsymbol{x_{10}}, \boldsymbol{s_1},\boldsymbol{x_{11}}, w_1],[ \boldsymbol{x_{20}}, \boldsymbol{s_2}, \boldsymbol{x_{21}}, w_2],G_2,\dots,G_n],\text{ and}\\
G&=[[\boldsymbol{x_{00}}, \boldsymbol{s_0}, \boldsymbol{x_{01}}, w_0],[ \boldsymbol{x_{10}}, \boldsymbol{s_1},\boldsymbol{x_{11}}, w_1],[ \boldsymbol{x_{20}}, \boldsymbol{s_2}, \boldsymbol{x_{21}}, w_2],G_2,\dots,G_n],
\end{split}
\end{equation}
where each $\boldsymbol{x_{ij}}$ is an appearance of one of the three crossings in question and $\boldsymbol{x_{i0}}\neq \boldsymbol{x_{i1}}$ for all $i=0,1,2$.  For clarity, we compare the versions of the notation presented in \eqref{E:Triangle1}-\eqref{E:triangle2}, as they pertain to Example \ref{Ex:8_18}:

\begin{equation}\label{E:Triangle3}
G=\left[\left[~\overbrace{\underbrace{1}_{\boldsymbol{a}}}^{\boldsymbol{x_{00}}}~,~
\overbrace{\underbrace{\color{white}~7~\color{black}}}_{\boldsymbol{s_0}}^{\boldsymbol{s_0}} ~\overbrace{\underbrace{2}_{\boldsymbol{b}}}^{\boldsymbol{x_{01}}}~,~ \overbrace{\underbrace{3, 4, 5, 6}_{w_0}}^{w_0}~,~ \overbrace{\underbrace{2}_{\boldsymbol{b}}}^{\boldsymbol{x_{10}}}~,~ \overbrace{\underbrace{\color{white}~7~\color{black}}_{\boldsymbol{s_1}}}^{\boldsymbol{s_1}} ~\overbrace{\underbrace{7}_{\boldsymbol{c}}}^{\boldsymbol{x_{11}}}~,~ \overbrace{\underbrace{4, 8, 6}_{w_1}}^{w_1}~,~ \overbrace{\underbrace{1}_{\boldsymbol{a}}}^{\boldsymbol{x_{20}}}~,~ \overbrace{\underbrace{\color{white}~7~\color{black}}_{\boldsymbol{s_2}}}^{\boldsymbol{s_2}} ~\overbrace{\underbrace{7}_{\boldsymbol{c}}}^{\boldsymbol{x_{21}}}~,~ \overbrace{\underbrace{3, 8, 5}_{w_2}}^{w_2}~\right]\right]
\end{equation}

In fact (we will not need this; we state it only for clarity), up to permutation and relabeling, there is really only one format of the third type, 
\[
[[\boldsymbol{a}, \boldsymbol{s_0}, \boldsymbol{b}, w_0],[ \boldsymbol{b}, \boldsymbol{s_1},\boldsymbol{c}, w_1],[ \boldsymbol{c}, \boldsymbol{s_2}, \boldsymbol{a}, w_2],G_4,\dots,G_n],
\]
two of the second type, 
\begin{align*}
&[[\boldsymbol{a}, \boldsymbol{s_0}, \boldsymbol{b}, w_0, \boldsymbol{a}, \boldsymbol{s_1},\boldsymbol{c}, w_1],[ \boldsymbol{b}, \boldsymbol{s_2}, \boldsymbol{c}, w_2],G_3,\dots,G_n] \text{ and }\\
&[[\boldsymbol{a}, \boldsymbol{s_0}, \boldsymbol{b}, w_0, \boldsymbol{c}, \boldsymbol{s_1},\boldsymbol{a}, w_1],[ \boldsymbol{b}, \boldsymbol{s_2}, \boldsymbol{c}, w_2],G_3,\dots,G_n],
\end{align*}
and four of the first type,
\begin{align*}
    &[[\boldsymbol{a}, \boldsymbol{s_0}, \boldsymbol{b}, w_0, \boldsymbol{b}, \boldsymbol{s_1}, \boldsymbol{c}, w_1, \boldsymbol{c}, \boldsymbol{s_2}, \boldsymbol{a}, w_2],G_2,\dots,G_n],\\
    &[[\boldsymbol{a}, \boldsymbol{s_0}, \boldsymbol{b}, w_0, \boldsymbol{c}, \boldsymbol{s_1}, \boldsymbol{a}, w_1, \boldsymbol{b}, \boldsymbol{s_2}, \boldsymbol{c}, w_2],G_2,\dots,G_n],\\
    &[[\boldsymbol{a}, \boldsymbol{s_0}, \boldsymbol{b}, w_0, \boldsymbol{b}, \boldsymbol{s_1}, \boldsymbol{c}, w_1, \boldsymbol{a}, \boldsymbol{s_2}, \boldsymbol{c}, w_2],G_2,\dots,G_n],\text{ and}\\
    &[[\boldsymbol{a}, \boldsymbol{s_0}, \boldsymbol{b}, w_0, \boldsymbol{c}, \boldsymbol{s_1}, \boldsymbol{a}, w_1, \boldsymbol{c}, \boldsymbol{s_2}, \boldsymbol{b}, w_2],G_2,\dots,G_n].
\end{align*}

The important thing to notice about each of these formats is that each pertinent component of the Gauss code follows the pattern (crossing, $\boldsymbol{s_i}$, crossing, $w_i$), possibly repeated.

\begin{figure}[H]
    \begin{center}
    \labellist
    \pinlabel{triangle smoothing} [l] at 100 -30
    \pinlabel{anti-triangle smoothing} [l] at 1450 -30
    \endlabellist
        \includegraphics[width=.3\textwidth]{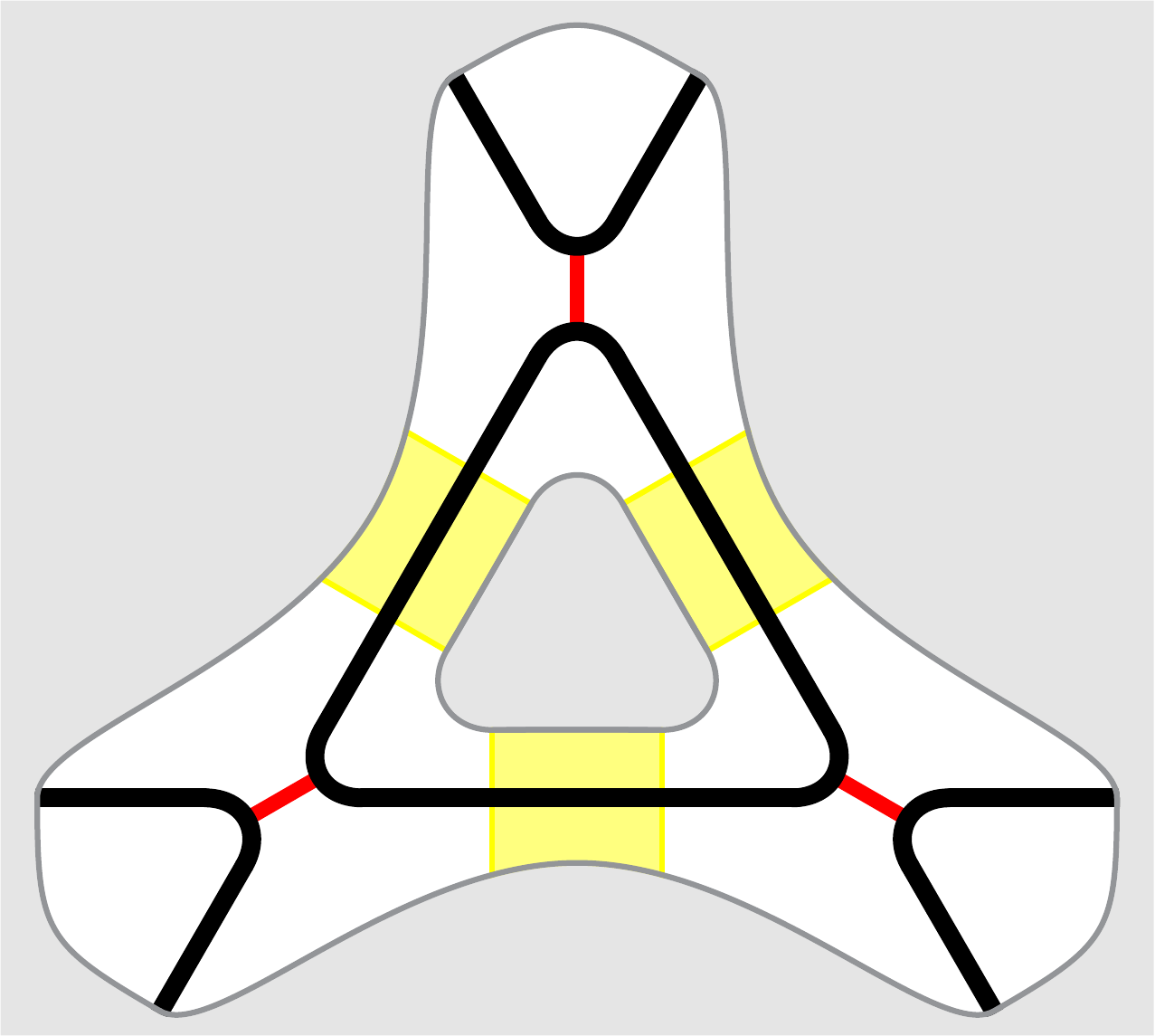}
        \raisebox{.75in}{$\leftarrow$}
        \includegraphics[width=.3\textwidth]{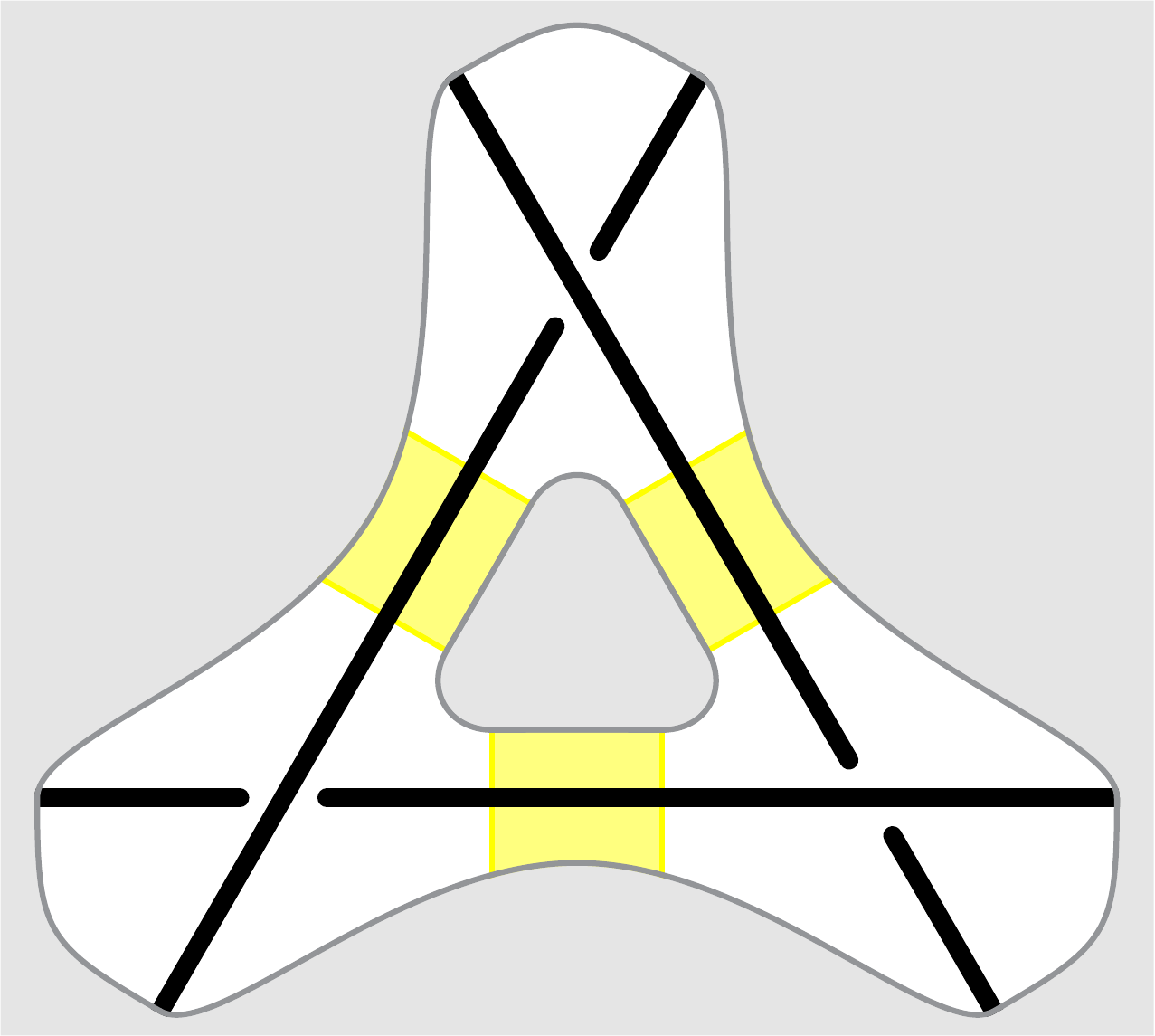}
        \raisebox{.75in}{$\rightarrow$}
        \includegraphics[width=.3\textwidth]{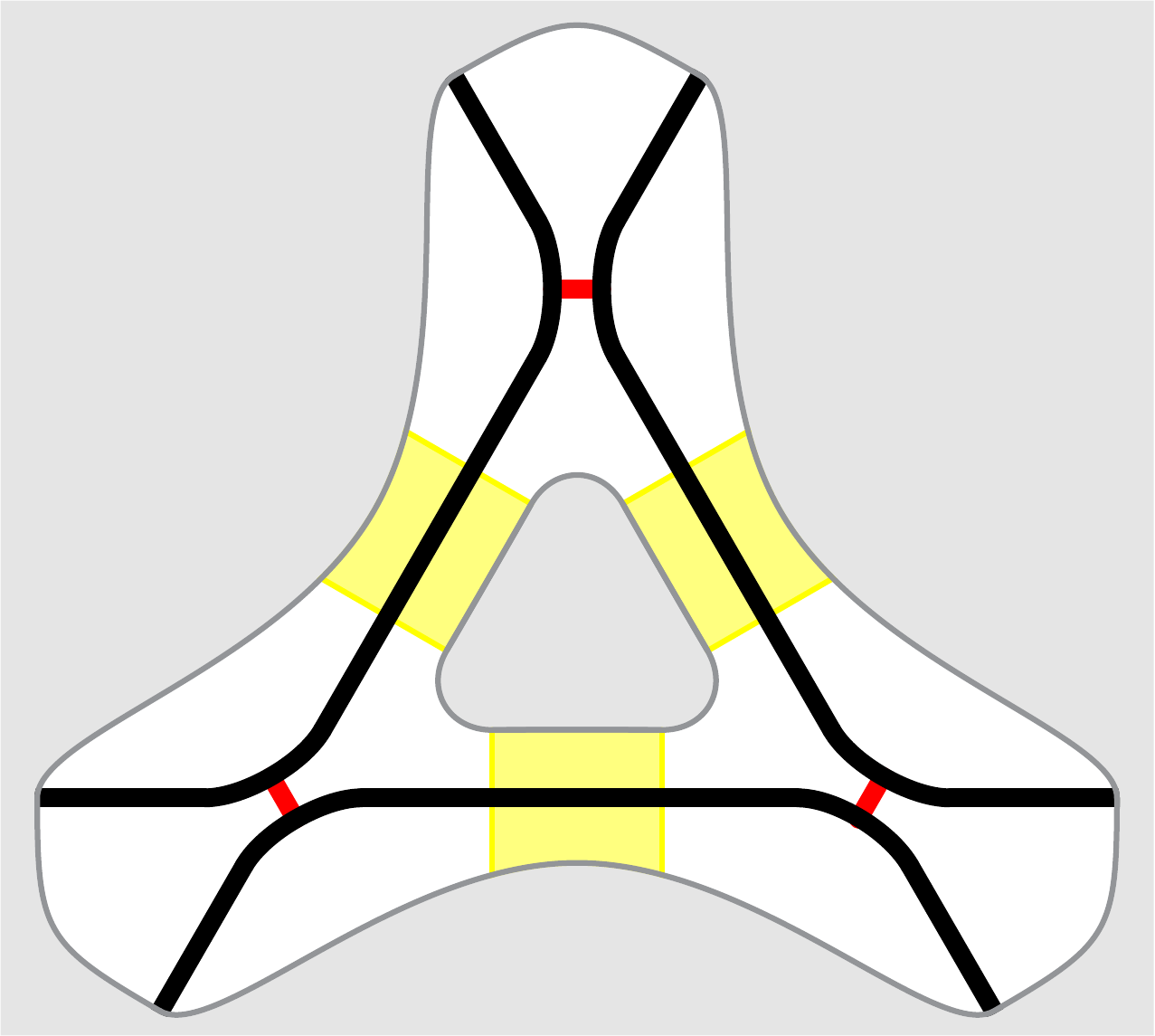}
        \caption{An arbitrary triangle (center) and its triangle and anti-triangle smoothings}
        \label{Fi:Triangle and Anti-triangle Smoothings}
    \end{center}
\end{figure}

Figure \ref{Fi:Triangle and Anti-triangle Smoothings} demonstrates what triangle and anti-triangle smoothings look like in general, and Figures \ref{Fi:8_18 Triangle Smoothing} and \ref{Fi:8_18 Anti-Triangle Smoothing} show how they look in Example \ref{Ex:8_18}. 

\begin{figure}[H]
    \begin{center}
        \labellist
        \small \hair 4pt
        \pinlabel{$p_1$} at 145 170
        \pinlabel{{\color{blue} $p_2$ }} at 70 280
        \pinlabel{1} at 127 203
        \pinlabel{2} at 195 288
        \pinlabel{3} at 320 150
        \pinlabel{4} at 235 115
        \pinlabel{5} at 115 100
        \pinlabel{6} at 20 180
        \pinlabel{7} at 220 230
        \pinlabel{8} at 155 20
        \endlabellist
        \includegraphics[height=.35\textwidth]{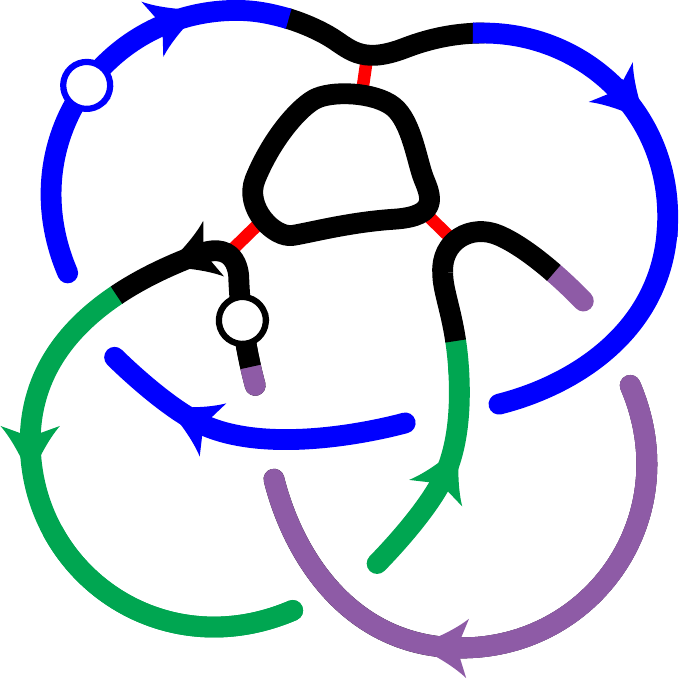}\hspace{.075\textwidth}
        \labellist
        \small \hair 4pt
        \pinlabel{1} at 135 205
        \pinlabel{2} at 530 205
        \pinlabel{7} at 335 540
        \pinlabel{{\color{blue} $w_0$} \color{black}} at 345 680
        \pinlabel{{\color{Purple} $w_1$}} at 590 40
        \pinlabel{{\color{Green} $w_2$}} at 70 40
        \tiny
        \pinlabel{$\boldsymbol{s_0}$} at 237 368
        \pinlabel{$\boldsymbol{s_1}$} at 335 205
        \pinlabel{$\boldsymbol{s_2}$} at 432 368
        \endlabellist
        \includegraphics[height=.4\textwidth]{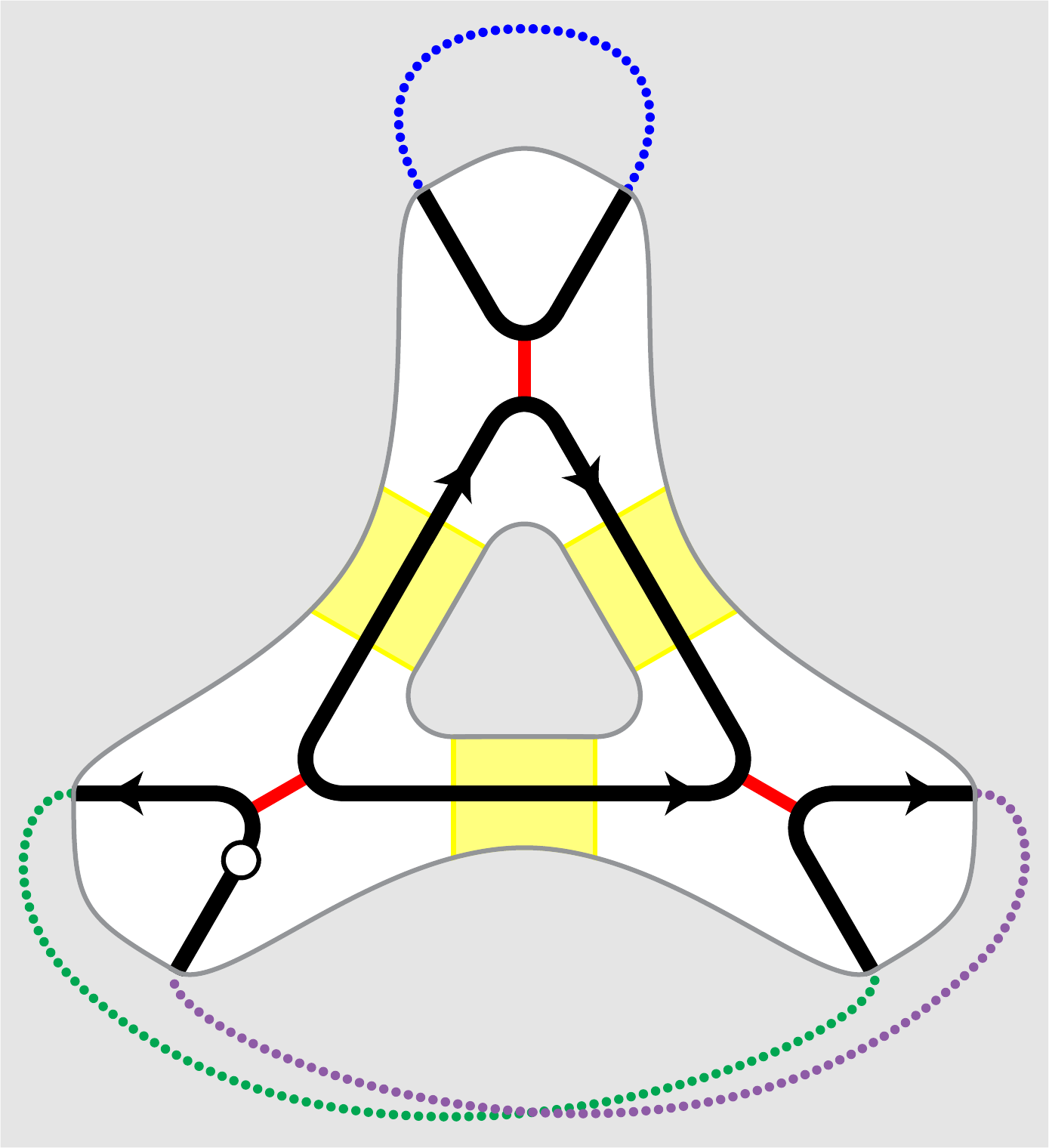}
        \caption{ The triangle smoothing for Example \ref{Ex:8_18} (also see Figure \ref{Fi:8_18})}
        \label{Fi:8_18 Triangle Smoothing}
    \end{center}
\end{figure}
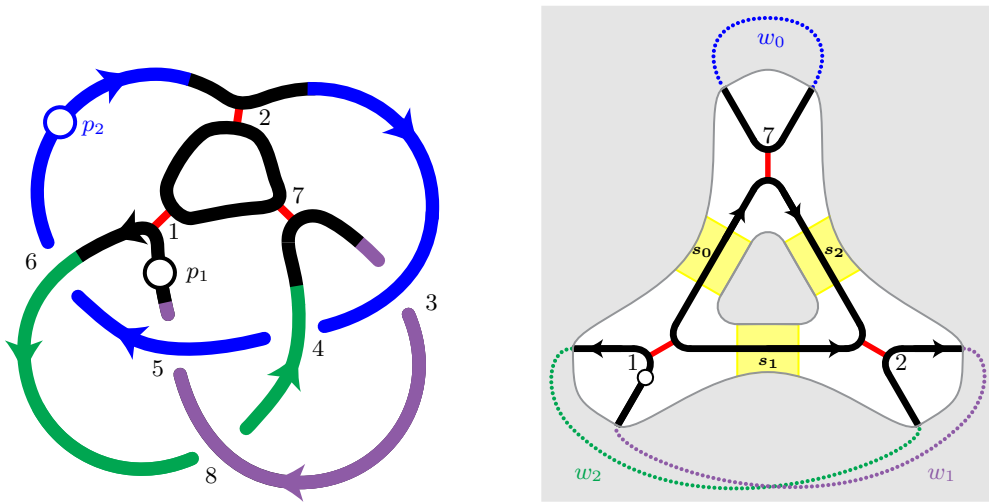

\subsubsection{Triangle smoothing.}
A {\it triangle smoothing} involves smoothing the triangle so that it becomes a state circle. The new state circle is formed by ``gluing'' the three segments $\boldsymbol{x_{i0},s_i,x_{i1}}$, $i=0,1,2$, at the three crossings.  Continuing Example \ref{Ex:8_18}, we glue the segments $\boldsymbol{1,2}$ and $\boldsymbol{2,7}$ and $\boldsymbol{1,7}$ at crossings $\boldsymbol{1}$, $\boldsymbol{2}$, and $\boldsymbol{7}$; the new state circle is $(\boldsymbol{1,2,7})$. We add all three crossings to our list of resolved crossings.  

We determine the Gauss code for a triangle smoothing as follows, using ``triangle pairings,'' each of which consists of a segment $w_i$ and the crossings right before and after it (cyclically) in the Gauss code.  For example, the triangle pairings in Example \ref{Ex:8_18} are $[\boldsymbol{b},w_0,\boldsymbol{b}]$, $[\boldsymbol{c},w_1,\boldsymbol{a}]$, and $[\boldsymbol{c},w_2,\boldsymbol{a}]$, where again $\boldsymbol{a}=\boldsymbol{1}$, $\boldsymbol{b}=\boldsymbol{2}$, $\boldsymbol{c}=\boldsymbol{7}$, $w_0=[3,4,5,6]$, $w_1=[4,8,6]$, and $w_2=[3,8,5]$. After writing down the three triangle pairings for a given triangle, we perform a triangle smoothing as follows.

\begin{enumerate}
    \item Start with the word $w_i$ in any unprocessed triangle pairing. 
    \item If the crossing $\boldsymbol{x_{j0}}$ after $w_i$ has not yet been written down, write $w_i$ and then $\boldsymbol{x_{j0}}$; otherwise, write $\overline{w_i}$ and then the crossing $\boldsymbol{x_{i1}}$ that immediately precedes $w_i$.  Mark this triangle pairing as processed. 
    \item Jump to the second occurrence of the crossing last written. 
    \item If the second occurrence of the crossing last written is in an unprocessed triangle pairing, repeat from Step (2). 
    \item If the second occurrence of the crossing last written is in a processed triangle pairing, close the component. If there are unprocessed triangle pairings left, start a new component and repeat from Step (1). If there are no unprocessed triangle pairings left, close the component because the smoothed Gauss code is done. 
\end{enumerate}

If we take the triangle pairings of \((\boldsymbol{2}, w_0, \boldsymbol{2})\), \((\boldsymbol{7}, w_1, \boldsymbol{1})\), and \((\boldsymbol{7}, w_2, \boldsymbol{1})\) from Example \ref{Ex:8_18}, we begin by starting with any unprocessed triangle pair. They are all unprocessed at this point, so let us start with the first one. 

Write down the word and the crossing after it from the pairing \((\boldsymbol{2}, w_0, \boldsymbol{2})\), mark this pairing as processed, jump to the second occurrence of $\boldsymbol{2}$, and, since it is in the triangle pairing that we just processed, close this component:

\[
[w_0, \boldsymbol{2}\to[w_0,\boldsymbol{2}].
\]

There are two unprocessed triangle pairings left, so start a new component. Choose the triangle pairing \((\boldsymbol{7}, w_1, \boldsymbol{1})\). Write down the word and the crossing after it, mark this pairing as processed, and jump to the second occurrence of $\boldsymbol{1}$, which is in the unprocessed triangle pairing \((\boldsymbol{7}, w_2, \boldsymbol{1})\). Since the crossing we landed on is at the end of the pairing, write the word in reverse and the crossing before it, then mark this pairing as processed, and jump to the second occurrence of $\boldsymbol{7}$. It is in the processed triangle pairing of \((\boldsymbol{7}, w_1, \boldsymbol{1})\) so close the component:

\[
[w_1, \boldsymbol{1}\quad\to\quad [w_1, \boldsymbol{1}, \overline{w_2}, \boldsymbol{7}\quad\to\quad [w_1, \boldsymbol{1}, \overline{w_2}, \boldsymbol{7}].
\]

Since no unprocessed triangle pairings remain, we are done. The format of the (changed part of the) final smoothed Gauss code is

\[
[[w_0, \boldsymbol{2}], [w_1, \boldsymbol{1}, \overline{w_2}, \boldsymbol{7}]],
\]
which agrees with Figure \ref{Fi:Triangle and Anti-triangle Smoothings}. Substituting for the $w_i$ and copying down any components of the previous Gauss code not involved with the triangle gives the new Gauss code.  In this example, 
this gives
\[
[[\boldsymbol{1},6,8,4,\boldsymbol{7},3,8,5],[\boldsymbol{2},3,4,5,6]],
\]
and so the resulting Gauss-state code is
\[
[[[\boldsymbol{1},6,8,4,\boldsymbol{7},3,8,5],[\boldsymbol{2},3,4,5,6]],[(\boldsymbol{1},\boldsymbol{2},\boldsymbol{7})],[\boldsymbol{1},\boldsymbol{2},\boldsymbol{7}]]
\]

\subsubsection{Anti-triangle smoothing.} An {\it anti-triangle smoothing} is the opposite of an triangle smoothing in the sense that it involves reversing all of the smoothings. 
This is the process of smoothing the anti-triangle:

\begin{enumerate}
    \item If all crossings have been written twice and all words and smoothed crossings have been written once, you are done. If this is not the case, start a new component and continue to Step (2) with a crossing 
    whose adjacent segment $\boldsymbol{s_i}$ of smoothed crossings has not been written. 
    \item Write down this crossing followed by the adjacent segment $\boldsymbol{s_i}$ of smoothed crossings; if $\boldsymbol{s_i}$ appears (cyclically) before the crossing, write down $\boldsymbol{\overline{s_i}}$, rather than $\boldsymbol{s_i}$.
    \item Look at the crossing on the opposite side of $\boldsymbol{s_i}$ and jump to its other occurrence. 
    \item Write down the crossing followed by the word $w_j$ next to it. The word may be on either side of the crossing. If the word is to the left, write the word in reverse order.  
    \item Look at the crossing on the opposite side of $w_j$ and jump to its other occurrence. 
    If you have already written down 
    the segment $\boldsymbol{s_k}$ of smoothed crossings next to it, close the component and repeat from Step (1). If this is not the case, continue from Step (2).
\end{enumerate}

This gives the format for the Gauss code that results from the anti-triangle smoothing.  Get the actual Gauss code from the format by substituting for the various $\boldsymbol{s_i}$ and $w_j$.

The idea here is to write the pattern of (crossing, smoothed crossings $\boldsymbol{s_i}$, crossing, word $w_j$) and repeat. You may have to traverse the Gauss code in the reverse direction, and that is fine, just reverse $\boldsymbol{s_i}$ or $w_j$.  

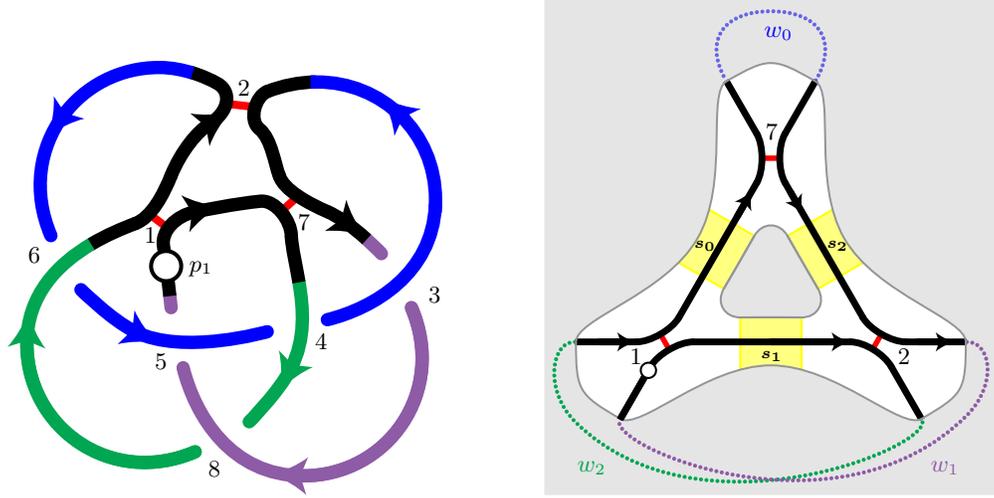
\begin{figure}[H]
    \begin{center}
        \labellist
        \small \hair 4pt
        \pinlabel{$p_1$} at 145 170
        \pinlabel{1} at 107 195
        \pinlabel{2} at 177 305
        \pinlabel{3} at 320 150
        \pinlabel{4} at 235 115
        \pinlabel{5} at 115 100
        \pinlabel{6} at 20 180
        \pinlabel{7} at 222 205
        \pinlabel{8} at 155 20
        \endlabellist
        \includegraphics[height=.35\textwidth]{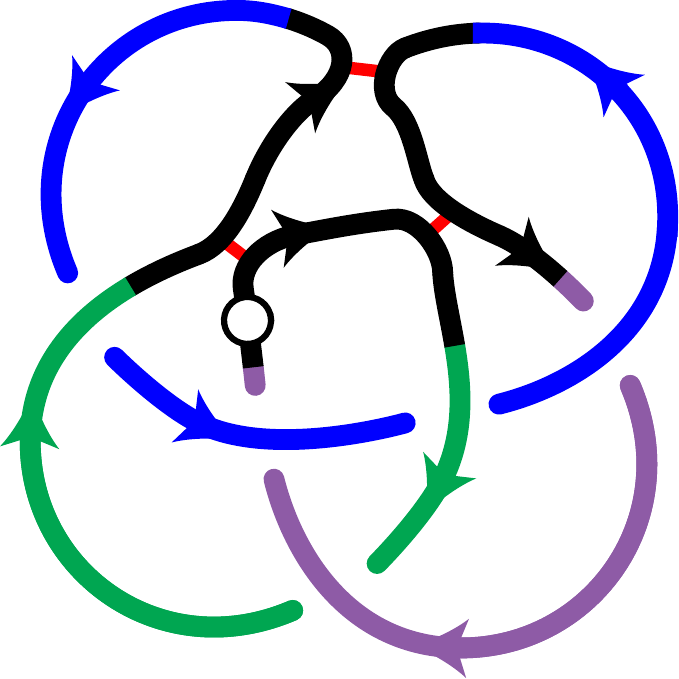}\hspace{.075\textwidth}
        \labellist
        \small \hair 4pt
        \pinlabel{1} at 135 205
        \pinlabel{2} at 530 205
        \pinlabel{7} at 335 535
        \pinlabel{{\color{blue} $w_0$}} at 345 680
        \pinlabel{{\color{Purple} $w_1$}} at 590 40
        \pinlabel{{\color{Green} $w_2$}} at 70 40
        \tiny
        \pinlabel{$\boldsymbol{s_0}$} at 237 368
        \pinlabel{$\boldsymbol{s_1}$} at 335 205
        \pinlabel{$\boldsymbol{s_2}$} at 432 368
        \endlabellist
        \includegraphics[height=.4\textwidth]{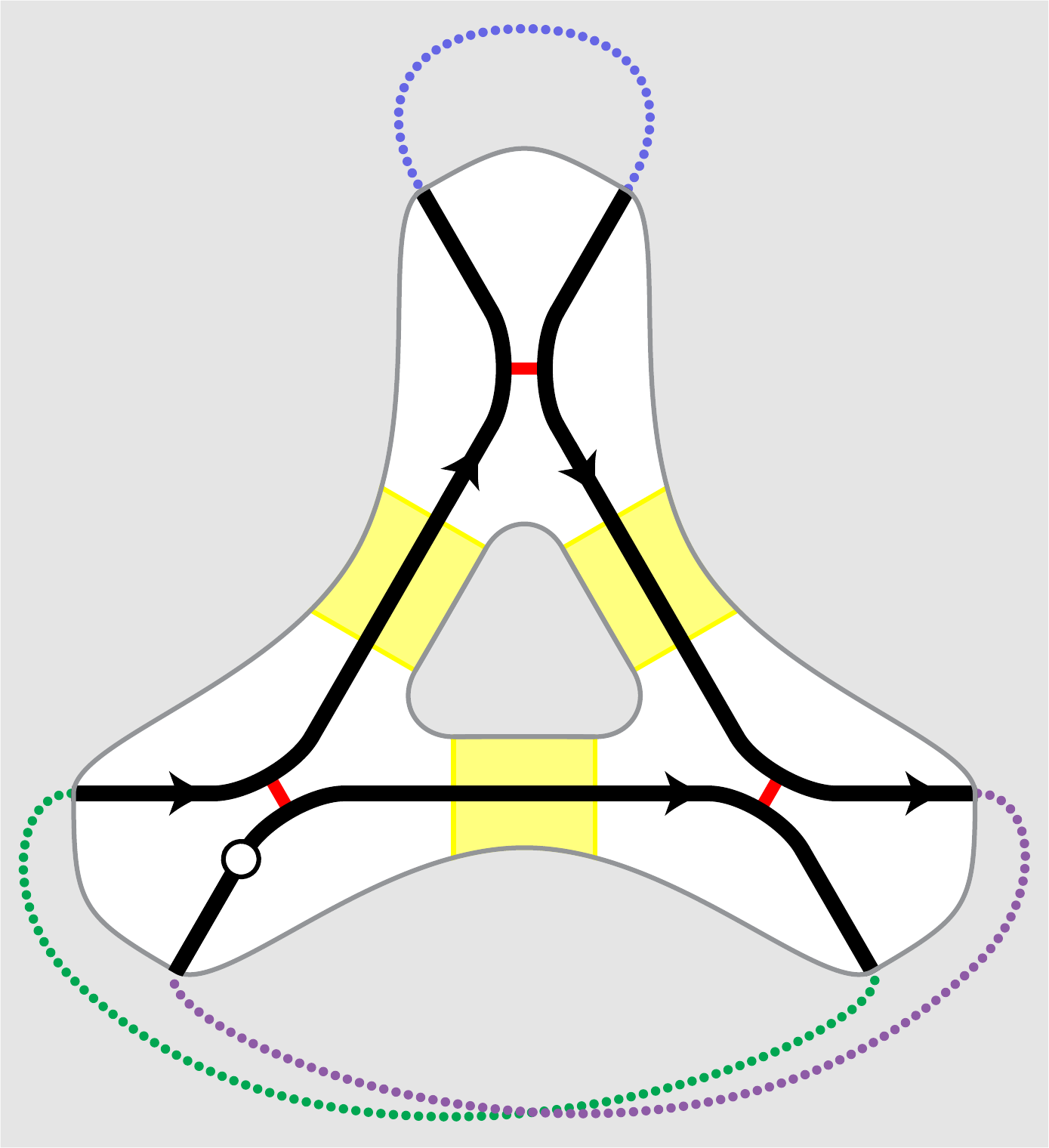}
        \caption{ The anti-triangle smoothing for Example \ref{Ex:8_18} (also see Figure \ref{Fi:8_18})}
        \label{Fi:8_18 Anti-Triangle Smoothing}
    \end{center}
\end{figure}

In Example \ref{Ex:8_18}, the process plays out like this, as we will explain below (also see Figure \ref{Fi:8_18 Anti-Triangle Smoothing}):
\begin{align*}
    &[[\boldsymbol{1}, \boldsymbol{s_0}\\
    \to& [[\boldsymbol{1}, \boldsymbol{s_0}, \boldsymbol{2}, \overline{w_0}\\
    \to& [[\boldsymbol{1}, \boldsymbol{s_0}, \boldsymbol{2}, \overline{w_0}, \boldsymbol{2}, \boldsymbol{s_1}\\
    \to& [[\boldsymbol{1}, \boldsymbol{s_0}, \boldsymbol{2}, \overline{w_0}, \boldsymbol{2}, \boldsymbol{s_1}, \boldsymbol{7}, w_2\\
    \to&[[\boldsymbol{1}, \boldsymbol{s_0}, \boldsymbol{2}, \overline{w_0}, \boldsymbol{2}, \boldsymbol{s_1}, \boldsymbol{7}, w_2, \boldsymbol{1}, \boldsymbol{s_2}\\
    \to&[[\boldsymbol{1}, \boldsymbol{s_0}, \boldsymbol{2}, \overline{w_0}, \boldsymbol{2}, \boldsymbol{s_1}, \boldsymbol{7}, w_2, \boldsymbol{1}, \boldsymbol{s_2}, \boldsymbol{7}, w_1\\
    \to&[[\boldsymbol{1}, \boldsymbol{s_0}, \boldsymbol{2}, \overline{w_0}, \boldsymbol{2}, \boldsymbol{s_1}, \boldsymbol{7}, w_2, \boldsymbol{1}, \boldsymbol{s_2}, \boldsymbol{7}, w_1]
\end{align*}
We begin by writing the first crossing, $\boldsymbol{1}$, and the adjacent smoothed crossings, $\boldsymbol{s_0}$. The next crossing is $\boldsymbol{2}$, so we jump to the other occurrence of $\boldsymbol{2}$ and write the word next to it, $w_0$, in this case, in reverse. The crossing on the other side of $w_0$ is the $\boldsymbol{2}$ before $w_0$, so we jump to the other occurrence of $\boldsymbol{2}$ (the one after $w_0$); since we have not yet written down the adjacent segment $\boldsymbol{s_1}$ of smoothed crossings, we write down $\boldsymbol{2}$ and $\boldsymbol{s_1}$. 
The next crossing is $\boldsymbol{7}$, so we jump to the other occurrence of $\boldsymbol{7}$ and write down $\boldsymbol{7}$ and the word next to it, $w_2$.
The crossing on the other side of $w_2$ is $\boldsymbol{1}$, so we jump to the other 
occurrence of $\boldsymbol{1}$ and consider the adjacent smoothed crossings $\boldsymbol{s_2}$; because $\boldsymbol{s_2}$ has not yet been written down, we write $\boldsymbol{1}$ and $\boldsymbol{s_2}$. The next crossing is $\boldsymbol{7}$, so we jump to the other 
occurrence of $\boldsymbol{7}$ and write down $\boldsymbol{7}$ and the adjacent word, $w_1$.
The next crossing is $\boldsymbol{1}$, so we jump to the other 
occurrence of $\boldsymbol{1}$ and consider the adjacent smoothed crossings $\boldsymbol{s_0}$. Since we have already written \([\boldsymbol{1}, \boldsymbol{s_0}, \boldsymbol{2}\), we close the component. All crossings have been written twice, and all words and smoothed crossings have been written once, so the format of the final smoothed Gauss code is
\[
[[\boldsymbol{1}, \boldsymbol{s_0}, \boldsymbol{2}, \overline{w_0}, \boldsymbol{2}, \boldsymbol{s_1}, \boldsymbol{7}, w_2, \boldsymbol{1}, \boldsymbol{s_2}, \boldsymbol{7}, w_1]].
\]
This agrees with Figure \ref{Fi:8_18 Anti-Triangle Smoothing}. Substituting for the $\boldsymbol{s_i}$ and $w_i$ gives the actual Gauss code that results from the anti-triangle smoothing in Example \ref{Ex:8_18}:

\[
[[\boldsymbol{1}, \boldsymbol{2}, 6,5,4,3, \boldsymbol{2}, \boldsymbol{7}, 3,8,5, \boldsymbol{1}, \boldsymbol{7}, 4,8,6]].
\]

\section{Results}\label{S:Results}

We developed a Python program that automates the minimal genus algorithm and computes the unoriented genus and crosscap number of an alternating knot or link from its Gauss code.  We then applied this program to all prime alternating knots through 19 crossings and all prime alternating links through 14 crossings. 

Next, 
we will describe how we automated Adams' and Kindred's minimal genus algorithm and computed unoriented genus and crosscap number. 
Then, 
we will summarize the data we obtained through our program. We will also describe alternative methods of performing these computations that we used to check the accuracy of our program.

\subsection{Automating the minimal genus algorithm}\label{S:Program}
We developed a Python program that automates the minimal genus algorithm.
The input to the program is a csv file of Gauss codes obtained from KnotInfo \cite{knotinfo}, LinkInfo \cite{linkinfo}, SnapPy \cite{snappy}, or Regina \cite{regina},\footnote{In the case of SnapPy and Regina, there was an intermediate step in which we converted DT codes taken directly from those repositories and converted them to Gauss codes.} and the output is a new csv file with the columns ``Knot" (or ``Link"), ``Gauss Code", ``State Code", ``Unoriented Genus", and ``Crosscap Number''. The columns ``Knot" (or ``Link") and ``Gauss Code" are populated with information from the input csv. For each knot, the program uses the given Gauss code to compute a state code for an optimal Kauffman state (which it records in the ``State Code'' column) and then determines the unoriented genus and crosscap number of the given knot or link, which it also records in the appropriate column. 

The program uses all detection and smoothing methods discussed in \textsection\ref{S:State codes}, along with the formulas and facts about unoriented genus and crosscap number from \textsection\ref{S:Background}. 
For each knot or link $K$ in the input csv, the program takes its given Gauss code $G=[G_1,\dots,G_n]$, and initiates a Gauss-state code $[G,[],[]]$. Then our program implements the following process.
 
 The following will not apply initially, but it may later as the program repeats this process: if all of the crossings in some component of $G$ have been smoothed, our program removes that component from $G$, adds it to the list of state circles, and restarts the process. 
 
 Next, our program first looks for a 1-gon. If it does find a 1-gon, it smooths it into a state circle, adds the smoothed crossing and new state circle to their respective lists, updates the Gauss code, and then restarts the process for this newer Gauss-state code. If a 1-gon is not detected, the program looks for a bigon. If a bigon is detected, it smooths it into a state circle, adds the smoothed crossings and new state circle to their respective lists, updates the Gauss code, and then restarts the process. If a bigon is not detected, the program looks for a triangle. When a triangle is detected, the program creates two copies of the current Gauss-state code and treats one via the triangle smoothing and the other via the anti-triangle smoothing. Then it separately  
 restarts the process with each newer Gauss-state code under consideration. 
 
The program repeats this process until every Gauss-state code is empty, at which point each list of state circles will be a state code.  From this list of state codes, the program chooses the longest one, i.e., the one with the most state circles.  It then computes the unoriented genus $\Gamma(K)=c+1-s$, where the crossing number $c$ equals $\frac{1}{2}\sum_i\text{length}(G_i)$, because each crossing appears exactly twice in $G$, and $s$ is the length of the longest state code, which equals the number of circles in the state $x$ that it describes.  

Finally, the program determines whether or not the associated state graph $G_x$ is simple and bipartite.  It is simple if and only if no pair of tuples in the state code for $x$ share more than one entry. It is bipartite if and only if the tuples in the state code for $x$ can be partitioned into two classes so that each crossing appears exactly once in each partition class. If $G_x$ is simple and bipartite, then $\gamma(K)=\Gamma(K)+1$, and if not then $\gamma(K)=\Gamma(K)$; this is due to Theorem \ref{T:CKLSV}.

We took several measures to bolster our confidence in the accuracy of our results.  First of all, the first and second authors independently coded everything and obtained the same results.  Secondly, for knots through 12 crossings, we checked that our results agree with the values listed on knotinfo.  Third, with knots with more crossings, we checked that our results agree with those obtained by a different method developed earlier by the second author \cite{tksplice}. That method, which employs splice-unknotting numbers, first introduced by Ito and Takimura \cite{it}, works for alternating knots but not for links of multiple components.  Fourth, we computed the crosscap numbers of all alternating knots and links through at least 12 crossings by a more exhaustive approach: find all adequate states and choose the shortest one, rather than employing the more efficient minimal genus algorithm.  The results obtained by all of these methods agree. 

Table \ref{T:10L} lists the crosscap numbers of all 10-crossing prime alternating links (with more than one component).  On LinkInfo, each such link is named L10a$k$ for some $k=1,\dots,174$; in the table, we denote that link by $10^a_k$. The vast majority of these links have unoriented genus equal to crosscap number, so we elect not to include separate columns for unoriented genus. Instead, the entry for $\gamma(L)$ is marked with a star in those rare cases where $\Gamma(L)=\gamma(L)-1$.  The rest have $\Gamma(L)=\gamma(L)$.

\begin{table}[H]
\begin{center}
\small
\begin{tabular}{|ll|ll|ll|ll|ll|ll|ll|ll|}
\hline\hline
L & $\gamma$ & L & $\gamma$ & L & $\gamma$ & L & $\gamma$ & L & $\gamma$ & L & $\gamma$ & L & $\gamma$ & L & $\gamma$ \\
\hline\hline
$10^a_{1}$ & 5 & $10^a_{2}$ & 4 & $10^a_{3}$ & 4 & $10^a_{4}$ & 4 & $10^a_{5}$ & 4 & $10^a_{6}$ & 5 & $10^a_{7}$ & 4 & $10^a_{8}$ & 4\\ \hline 
$10^a_{9}$ & 3 & $10^a_{10}$ & 4 & $10^a_{11}$ & 5 & $10^a_{12}$ & 5 & $10^a_{13}$ & 4 & $10^a_{14}$ & 5 & $10^a_{15}$ & 4 & $10^a_{16}$ & 3\\ \hline 
$10^a_{17}$ & 5 & $10^a_{18}$ & 4 & $10^a_{19}$ & 4 & $10^a_{20}$ & 4 & $10^a_{21}$ & 5 & $10^a_{22}$ & 4 & $10^a_{23}$ & 4 & $10^a_{24}$ & 4\\ \hline 
$10^a_{25}$ & 4 & $10^a_{26}$ & 4 & $10^a_{27}$ & 3 & $10^a_{28}$ & 4 & $10^a_{29}$ & 4 & $10^a_{30}$ & 5 & $10^a_{31}$ & 3 & $10^a_{32}$ & 4\\ \hline 
$10^a_{33}$ & 3 & $10^a_{34}$ & 5 & $10^a_{35}$ & 4 & $10^a_{36}$ & 4 & $10^a_{37}$ & 5 & $10^a_{38}$ & 5 & $10^a_{39}$ & 4 & $10^a_{40}$ & 4$^*$\\ \hline 
$10^a_{41}$ & 4 & $10^a_{42}$ & 5 & $10^a_{43}$ & 5 & $10^a_{44}$ & 3 & $10^a_{45}$ & 4 & $10^a_{46}$ & 3 & $10^a_{47}$ & 4 & $10^a_{48}$ & 3\\ \hline 
$10^a_{49}$ & 4 & $10^a_{50}$ & 4$^*$ & $10^a_{51}$ & 5 & $10^a_{52}$ & 4 & $10^a_{53}$ & 4 & $10^a_{54}$ & 4 & $10^a_{55}$ & 4 & $10^a_{56}$ & 5\\ \hline 
$10^a_{57}$ & 4 & $10^a_{58}$ & 4 & $10^a_{59}$ & 4 & $10^a_{60}$ & 4 & $10^a_{61}$ & 5 & $10^a_{62}$ & 3 & $10^a_{63}$ & 4 & $10^a_{64}$ & 4\\ \hline 
$10^a_{65}$ & 4 & $10^a_{66}$ & 3 & $10^a_{67}$ & 3 & $10^a_{68}$ & 4 & $10^a_{69}$ & 5 & $10^a_{70}$ & 4 & $10^a_{71}$ & 5 & $10^a_{72}$ & 4\\ \hline 
$10^a_{73}$ & 3 & $10^a_{74}$ & 3 & $10^a_{75}$ & 3 & $10^a_{76}$ & 5 & $10^a_{77}$ & 4 & $10^a_{78}$ & 4$^*$ & $10^a_{79}$ & 4 & $10^a_{80}$ & 3\\ \hline 
$10^a_{81}$ & 4 & $10^a_{82}$ & 5 & $10^a_{83}$ & 5 & $10^a_{84}$ & 4 & $10^a_{85}$ & 5 & $10^a_{86}$ & 5 & $10^a_{87}$ & 5 & $10^a_{88}$ & 4\\ \hline 
$10^a_{89}$ & 3 & $10^a_{90}$ & 4 & $10^a_{91}$ & 5 & $10^a_{92}$ & 4 & $10^a_{93}$ & 4 & $10^a_{94}$ & 4 & $10^a_{95}$ & 4 & $10^a_{96}$ & 4\\ \hline 
$10^a_{97}$ & 3 & $10^a_{98}$ & 3 & $10^a_{99}$ & 3 & $10^a_{100}$ & 3 & $10^a_{101}$ & 4$^*$ & $10^a_{102}$ & 3 & $10^a_{103}$ & 4 & $10^a_{104}$ & 5\\ \hline 
$10^a_{105}$ & 3 & $10^a_{106}$ & 4 & $10^a_{107}$ & 5 & $10^a_{108}$ & 5 & $10^a_{109}$ & 4 & $10^a_{110}$ & 3 & $10^a_{111}$ & 5 & $10^a_{112}$ & 4\\ \hline 
$10^a_{113}$ & 3 & $10^a_{114}$ & 2 & $10^a_{115}$ & 3 & $10^a_{116}$ & 4 & $10^a_{117}$ & 3 & $10^a_{118}$ & 2$^*$ & $10^a_{119}$ & 4$^*$ & $10^a_{120}$ & 2\\ \hline 
$10^a_{121}$ & 5 & $10^a_{122}$ & 4 & $10^a_{123}$ & 4 & $10^a_{124}$ & 5 & $10^a_{125}$ & 4 & $10^a_{126}$ & 4 & $10^a_{127}$ & 5 & $10^a_{128}$ & 4\\ \hline 
$10^a_{129}$ & 5 & $10^a_{130}$ & 4 & $10^a_{131}$ & 4 & $10^a_{132}$ & 4 & $10^a_{133}$ & 4 & $10^a_{134}$ & 4 & $10^a_{135}$ & 4 & $10^a_{136}$ & 5\\ \hline 
$10^a_{137}$ & 4 & $10^a_{138}$ & 3 & $10^a_{139}$ & 4 & $10^a_{140}$ & 4 & $10^a_{141}$ & 3 & $10^a_{142}$ & 3 & $10^a_{143}$ & 4 & $10^a_{144}$ & 3\\ \hline 
$10^a_{145}$ & 3$^*$ & $10^a_{146}$ & 4 & $10^a_{147}$ & 4 & $10^a_{148}$ & 4 & $10^a_{149}$ & 5 & $10^a_{150}$ & 4 & $10^a_{151}$ & 5 & $10^a_{152}$ & 4\\ \hline 
$10^a_{153}$ & 4 & $10^a_{154}$ & 4 & $10^a_{155}$ & 5$^*$ & $10^a_{156}$ & 5 & $10^a_{157}$ & 3 & $10^a_{158}$ & 4 & $10^a_{159}$ & 4 & $10^a_{160}$ & 3\\ \hline 
$10^a_{161}$ & 3$^*$ & $10^a_{162}$ & 4 & $10^a_{163}$ & 4 & $10^a_{164}$ & 4 & $10^a_{165}$ & 5 & $10^a_{166}$ & 4 & $10^a_{167}$ & 4$^*$ & $10^a_{168}$ & 5\\ \hline 
$10^a_{169}$ & 5 & $10^a_{170}$ & 4 & $10^a_{171}$ & 4 & $10^a_{172}$ & 4$^*$ & $10^a_{173}$ & 5 & $10^a_{174}$ & 5$^*$ &  &  &  &  \\
\hline\hline
\end{tabular}
\caption{The crosscap number $\gamma(L)$ of each prime alternating link $L$ with 10 crossings. Those marked with a star have unoriented genus $\Gamma(L)=\gamma(L)-1$; the rest have $\Gamma(L)=\gamma(L)$.
}\label{T:10L}
\end{center}
\end{table}

\subsection{Summary of data}

We computed the unoriented genus and crosscap number of every prime alternating knot through 19 crossings and every prime alternating link through 14 crossings; this data is new for knots with 14-19 crossings and for links with 10-14 crossings \cite{ak,tksplice}.  
In this final subsection, we present digested summaries of our data and examine intriguing patterns which emerge.  More data, along with our code, will be available at \href{https://github.com/IsaiasBahena/minimal-genus-state-surfaces}{this repository}.

First, Table \ref{T:UGKCC} lists the number of prime alternating knots with a given crossing number $c\leq 19$ that have a given unoriented genus $\Gamma$ or crosscap number $\gamma$. The table also shows the mean, mean, mode, and maximum values of $\Gamma$ and $\gamma$ at each crossing number $c\leq 19$.  
Table \ref{T:UGLCC} lists the same information as Table \ref{T:UGKCC}, but for prime alternating links with $c\leq 14$.
Tables \ref{T:CrosscapKStacked}-\ref{T:CrosscapLStacked} summarize the data from Tables \ref{T:UGKCC}-\ref{T:UGLCC} visually in stacked bar charts.

\begin{table}[H]
\begin{center}
    \small
    $\begin{array}{| cc |ccccccccc|cccc|}
\hline\hline
c&\gamma,\Gamma& 1 & 2 & 3 & 4 & 5 & 6 & 7 & 8 & 9 & \text{mean} & \text{median} & \text{mode} & \text{max} \\
\hline\hline
\multirow{2}{4pt}{3}  & \Gamma & 1 &  &  &  &  &  &  &  &  & 1 & 1 & 1 & 1 \\ 
& \gamma & 1 &  &  &  &  &  &  &  &  & 1 & 1 & 1 & 1 \\ \hline
\multirow{2}{4pt}{4}  & \Gamma &    &  1 &  &  &  &  &  &  &  & 2 & 2 & 2 & 2 \\ 
  & \gamma &  &  1 &  &  &  &  &  &  &  & 2 & 2 & 2 & 2 \\ \hline
\multirow{2}{4pt}{5}  & \Gamma & 1  &  1 &  &  &  &  &  &  &  & 1.5 & 1.5 & 1,2 & 2 \\ 
  & \gamma &1  &  1 &  &  &  &  &  &  &  & 1.5 & 1.5 & 1,2 & 2 \\ \hline
\multirow{2}{4pt}{6}  & \Gamma &    &  2  &  1 &  &  &  &  &  &  & 2.33 & 2 & 2 & 3 \\ 
  & \gamma &    &  2  &  1 &  &  &  &  &  &  & 2.33 & 2 & 2 & 3 \\ \hline
\multirow{2}{4pt}{7}  & \Gamma & 1  &  3  &  3 &  &  &  &  &  &  & 2.29 & 2 & 2,3 & 3 \\ 
  & \gamma & 1  &  2  &  4 &  &  &  &  &  &  & 2.43 & 3 & 3 & 3 \\ \hline
\multirow{2}{4pt}{8}  & \Gamma &    &  5  &  8  &  5 &  &  &  &  &  & 3 & 3 & 3 & 4 \\ 
  & \gamma &    &  4  &  9  &  5 &  &  &  &  &  & 3.06 & 3 & 3 & 4 \\ \hline
\multirow{2}{4pt}{9}  & \Gamma & 1  &  5  &  16  &  19 &  &  &  &  &  & 3.29 & 3 & 4 & 4 \\ 
  & \gamma &  1  &  3  &  18  &  19 &  &  &  &  &  & 3.34 & 3 & 4 & 4 \\ \hline
\multirow{2}{4pt}{10}  & \Gamma &    &  7  &  28  &  68  &  20 &  &  &  &  & 3.82 & 4 & 4 & 5 \\ 
  & \gamma  &    &  6  &  29  &  67  &  21 &  &  &  &  & 3.84 & 4 & 4 & 5 \\ \hline
\multirow{2}{4pt}{11}  & \Gamma & 1  &  7  &  45  &  155  &  159 &  &  &  &  & 4.26 & 4 & 5 & 5 \\ 
  & \gamma &  1  &  4  &  48  &  153  &  161 &  &  &  &  & 4.28 & 4 & 5 & 5 \\ \hline
\multirow{2}{4pt}{12}  & \Gamma &    &  11  &  69  &  371  &  622  &  215 &  &  &  & 4.75 & 5 & 5 & 6 \\ 
  & \gamma &     &  9  &  71  &  363  &  630  &  215 &  &  &  & 4.75 & 5 & 5 & 6 \\ \hline
\multirow{2}{4pt}{13}  & \Gamma & 1  &  9  &  102  &  689  &  2151  &  1926 &  &  &  & 5.21 & 5 & 5 & 6 \\ 
  & \gamma &  1  &  5  &  106  &  670  &  2170  &  1926 &  &  &  & 5.21 & 5 & 5 & 6 \\ \hline
\multirow{2}{4pt}{14}  & \Gamma &    &  14  &  141  &  1360  &  5466  &  9987  &  2568 &  &  & 5.69 & 6 & 6 & 7 \\
  & \gamma &     &  12  &  143  &  1318  &  5508  &  9984  &  2571 &  &  & 5.69 & 6 & 6 & 7 \\ \hline
\multirow{2}{4pt}{15}  & \Gamma & 1  &  12  &  202  &  2254  &  14075  &  38266  &  30453 &  &  & 6.13 & 6 & 6 & 7 \\ 
  &  \gamma  & 1  &  6  &  208  &  2169  &  14160  &  38241  &  30478 &  &  & 6.13 & 6 & 6 & 7 \\ \hline
\multirow{2}{4pt}{16}  & \Gamma &    &  19  &  258  &  4002  &  29432  &  121351  &  182086  &  42651 &  & 6.60 & 7 & 7 & 8 \\ 
& \gamma &     &  16  &  261  &  3852  &  29582  &  121219  &  182218  &  42651 &  & 6.60 & 7 & 7 & 8 \\ \hline
\multirow{2}{4pt}{17}  & \Gamma & 1  &  14  &  363  &  6170  &  64577  &  339932  &  805856  &  553066 &  & 7.04 & 7 & 7 & 8 \\ 
  & \gamma &  1  &  7  &  370  &  5892  &  64855  &  339346  &  806442  &  553066 &  & 7.04 & 7 & 7 & 8 \\ \hline
\multirow{2}{4pt}{18} & \Gamma &    &  23  &  436  &  10211  &  120190  &  862745  &  2885155  &  3772215  &  749310 & 7.49 & 8 & 8 & 9 \\ 
  & \gamma &     &  20  &  439  &  9776  &  120625  &  860676  &  2887224  &  3772173  &  749352 & 7.49 & 8 & 8 & 9 \\ \hline
\multirow{2}{4pt}{19}  & \Gamma  & 1  &  17  &  612  &  14909  &  237825  &  2027386  &  9108976  &  18308209  &  10921450 & 7.93 & 8 & 8 & 9 \\ 
  &  \gamma & 1  &  8  &  621  &  14167  &  238567  &  2020991  &  9115371  &  18307751  &  10921908 & 7.93 & 8 & 8 & 9 \\ \hline
\hline
    \end{array}$
    \caption{The number of prime alternating knots with a given crossing number $c\leq 19$ and unoriented genus $\Gamma$ or crosscap number $\gamma$}
    \label{T:UGKCC}
\end{center}
\end{table}

Recall from Proposition \ref{P:c2} that the maximum values of crosscap number and unoriented genus among prime alternating knots and links with a given crossing number $c\geq 3$ are both at most $\frac{c}{2}$, in fact at most $\lfloor \frac{c}{2} \rfloor$. Notice in Tables \ref{T:UGKCC}-\ref{T:UGLCC} that this bound is always achieved, with a single exception: unoriented genus for links with four crossings. Notice further that among all prime alternating knots with each crossing number $8\leq c\leq 19$, the median value (shared by $\gamma$ and $\Gamma$) is always one less than the maximum value of $\lfloor \frac{c}{2}\rfloor$, and when $c\geq 12$ this value of $\lfloor \frac{c}{2}\rfloor -1$ is also the (shared) mode.  We conjecture that this behavior persists.

\begin{conjecture}
For any $c\geq 12$, the medians and modes for $\gamma$ and $\Gamma$ among all $c$-crossing prime alternating knots equal $\lfloor\frac{c}{2}\rfloor-1$.
\end{conjecture}

It is also interesting to note that among knots earlier in Table \ref{T:UGKCC}, this phenomenon arises more quickly among knots with even, rather than odd, crossing number.  Parity of crossing number also appears to be significant in the columns $\gamma=1,2$ and $\Gamma=1,2$, which (unlike the other columns, apparently) do not increase monotonically. 

The pattern $1,0,1,0,1,\dots$ in the columns $\gamma=1$ and $\Gamma=1$ in Table \ref{T:UGKCC} persists indefinitely, because the only alternating knots with $\Gamma=1$ are the $(2,q)$-torus knots, each of which bounds an unknotted M\"obius band. No alternating knot or link can bound a knotted M\"obius band or annulus $F$ because $\partial \nu F$ would be an essential torus in the knot or link complement, but every alternating link which is not a $(2,q)$-torus knot or link is hyperbolic \cite{men84}.

\begin{table}[H]
\begin{center}
    \small
    $\begin{array}{| cc |ccccccccc|cccc|}
\hline\hline
c&\gamma,\Gamma& 1 & 2 & 3 & 4 & 5 & 6 & 7 & 8 & 9 & \text{mean} & \text{median} & \text{mode} & \text{max} \\
\hline\hline
\multirow{2}{4pt}{2}  & \Gamma & 1 &  &  &  &  &  &  &  &  & 1 & 1 & 1 & 1 \\ 
& \gamma &  & 1 &  &  &  &  &  &  &  & 2 & 2 & 2 & 2 \\ \hline
\multirow{2}{4pt}{4}  & \Gamma & 1 &  &  &  &  &  &  &  &  & 1 & 1 & 1 & 1 \\ 
& \gamma &  & 1 &  &  &  &  &  &  &  & 2 & 2 & 2 & 2 \\ \hline
\multirow{2}{4pt}{5}  & \Gamma &  & 1 &  &  &  &  &  &  &  & 2 & 2 & 2 & 2 \\ 
& \gamma &  & 1 &  &  &  &  &  &  &  & 2 & 2 & 2 & 2 \\ \hline
\multirow{2}{4pt}{6}  & \Gamma & 1  & 2 & 2 &  &  &  &  &  &  & 2.2 & 2 & 2,3 & 3 \\ 
  & \gamma &  & 2 & 3 &  &  &  &  &  &  & 2.6 & 3 & 3 & 3 \\ \hline
\multirow{2}{4pt}{7}  & \Gamma &  & 2 & 5 &  &  &  &  &  &  & 2.71 & 3 & 3 & 3\\
  & \gamma &  & 2 & 5 &  &  &  &  &  &  & 2.71 & 3 & 3 & 3 \\ \hline
\multirow{2}{4pt}{8}  & \Gamma & 1 & 2 & 12 & 6 &  &  &  &  &  & 3.10 & 3 & 3 & 4 \\ 
  & \gamma &  & 2 & 11 & 8 &  &  &  &  &  & 3.29 & 3 & 3 & 4 \\ \hline
\multirow{2}{4pt}{9}  & \Gamma &  & 4 & 22 & 29 &  &  &  &  &  & 3.45 & 4 & 4 & 4 \\ 
  & \gamma &  & 4 & 20 & 31 &  &  &  &  &  & 3.49 & 4 & 4 & 4 \\ \hline
\multirow{2}{4pt}{10}  & \Gamma & 1 & 4 & 39 & 89 & 41 &  &  &  &  & 3.95 & 4 & 4 & 5 \\
  & \gamma &  & 3 & 34 & 94 & 43 &  &  &  &  & 4.02 & 4 & 4 & 5 \\ \hline
\multirow{2}{4pt}{11} & \Gamma &  & 6 & 60 & 230 & 252 &  &  &  &  & 4.33 & 5 & 5 & 5 \\
  & \gamma &  & 6 & 52 & 233 & 257 &  &  &  &  & 4.35 & 5 & 5 & 5 \\ \hline
\multirow{2}{4pt}{12} & \Gamma & 1 & 5 & 95 & 482 & 1103 & 334 &  &  &  & 4.82 & 5 & 5 & 6 \\ 
  & \gamma &  & 3 & 77 & 485 & 1112 & 343 &  &  &  & 4.85 & 5 & 5 & 6 \\ \hline
\multirow{2}{4pt}{13}  & \Gamma &  & 9 & 133 & 989 & 3264 & 3144 &  &  &  & 5.25 & 5 & 5 & 6 \\ 
  & \gamma &  & 9 & 111 & 977 & 3276 & 3166 &  &  &  & 5.26 & 5 & 5 & 6 \\ \hline
\multirow{2}{4pt}{14}  & \Gamma & 1 & 7 & 195 & 1766 & 9153 & 16208 & 4481 &  &  & 5.72 & 6 & 6 & 7 \\
  & \gamma &  & 4 & 156 & 1725 & 9120 & 16310 & 4496 &  &  & 5.73 & 6 & 6 & 7 \\ \hline
\hline
    \end{array}$
    \caption{The number of prime alternating links with a given crossing number $c\leq 14$ and unoriented genus $\Gamma$ or crosscap number $\gamma$}
    \label{T:UGLCC}
\end{center}
\end{table}

\begin{table}[H]
\begin{center}
    \includegraphics[width=5in]{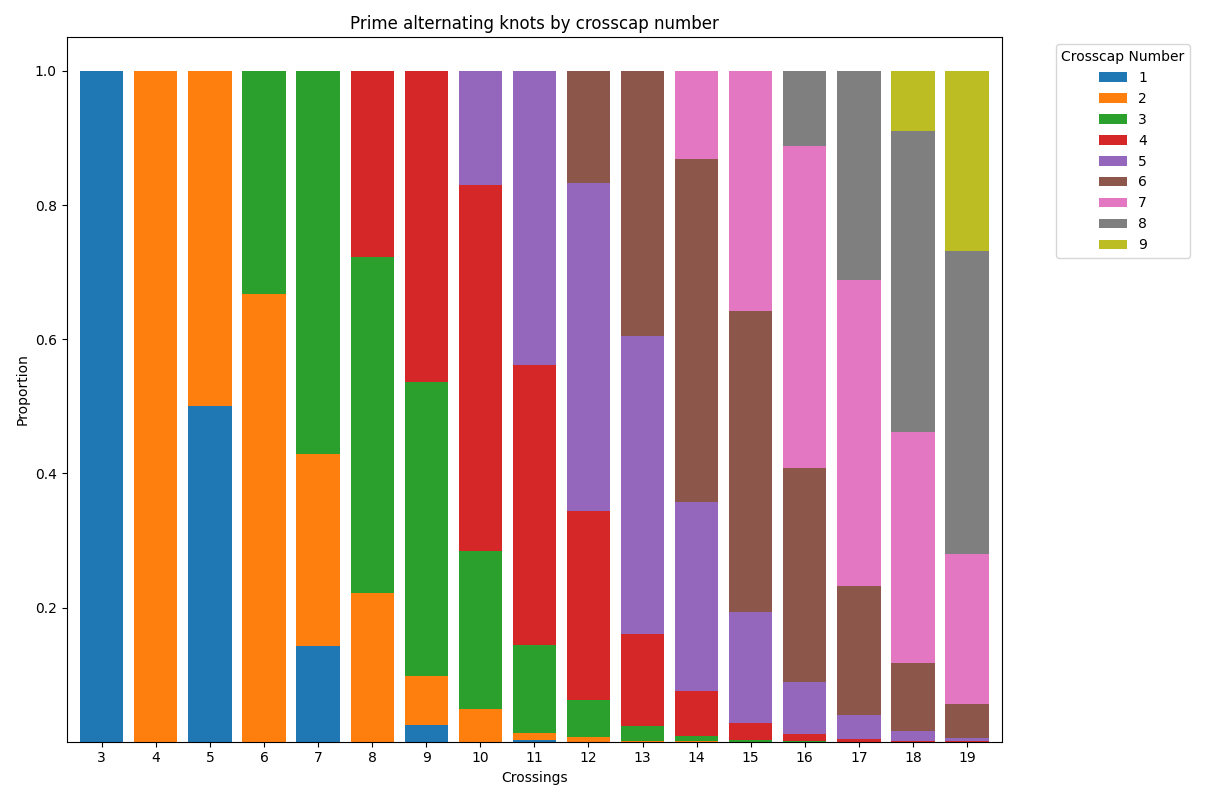}
    \caption{The proportion of crosscap numbers of prime alternating knots, by crossing number 
    }
    \label{T:CrosscapKStacked}
\end{center}
\end{table}

\begin{table}[H]
\begin{center}
    \includegraphics[width=5in]{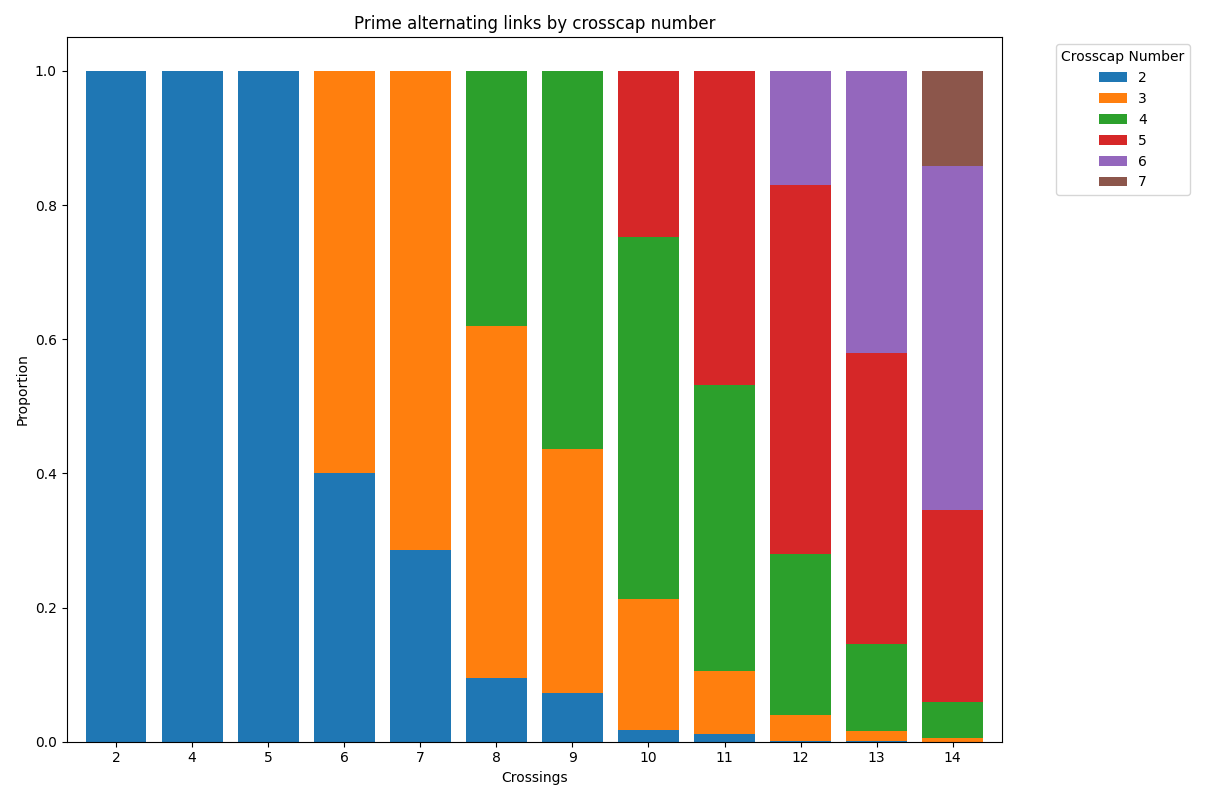}
    \caption{The proportion of crosscap numbers of prime alternating links, by crossing number 
    }
    \label{T:CrosscapLStacked}
\end{center}
\end{table}

We conjecture that several observable patterns in Table \ref{T:UGKCC} related to the monotonicity of the other columns persist indefinitely.

\begin{conjecture}\label{C:Monotonic}
Given $k\in\mathbb{Z}^+$, let $S_{\gamma=k}$ and $S_{\Gamma=k}$ respectively denote the set of prime alternating knots with $\gamma=k$ and $\Gamma=k$. 
\begin{enumerate}
    \item For any even $c\geq 8$, $S_{\gamma=2}$ contains \emph{fewer} $(c+1)$-crossing knots than $c$-crossing knots. The same is true of $S_{\Gamma=2}$ for all even $c\geq 12$. (This is counterintuitive.)
    \item For any $c\geq 3$, $S_{\gamma=2}$ and $S_{\Gamma=2}$ each contain more $(c+2)$-crossing knots than $c$-crossing knots.
    \item For any $c\geq 6$ and any $3\leq k\leq \frac{c}{2}$, $S_{\gamma=k}$ and $S_{\Gamma=k}$ each contain more $c$-crossing knots than $(c-1)$-crossing knots.
   \end{enumerate}
\end{conjecture}

For prime alternating links, we observe similar behaviors in Table~\ref{T:UGLCC} for crosscap number and unoriented genus.  
The distinctive patterns observed for knots with $\gamma=2$ or $\Gamma=2$, and described in parts (1) and (2) of Conjecture \ref{C:Monotonic}, is less pronounced for links; it is less clear what behavior to expect for such links as crossing number increases. Filtering the data by number of link components might add clarity here.
All frequencies for links with $\gamma\geq 3$ or $\Gamma \geq 3$ in Table \ref{T:UGLCC} increase with crossing number, and we conjecture that this behavior persists.

It seems feasible that one could prove the first two parts of Conjecture \ref{C:Monotonic}, at least for $\gamma$, by using the classification of prime alternating knots with $\gamma=2$ provided by Ito and Takimura in \cite{it}.  Perhaps one could also prove the third part of the conjecture, at least for $\gamma=3$, using the classification of prime alternating knots with $\gamma=3$ provided by Ito and Takimura in \cite{it3}.

A particularly striking observation in Tables \ref{T:UGKCC} and \ref{T:UGLCC} is just how uncommon it is for a knot or link to have $\gamma>\Gamma$. Rather, the vast majority of knots and links in the tables satisfy $\gamma=\Gamma$, meaning that the unoriented genus is realized by a 1-sided surface.

Tables \ref{T:DefectK} and \ref{T:DefectL} exhibit this phenomenon explicitly by enumerating how many prime alternating knots and links, respectively, have crosscap number $\gamma>\Gamma$. The tables also indicate the proportions at which such examples occur, by crossing number.  
Table~\ref{T:DefectProp} plots these proportions on a log-linear scale.  Observe that the data suggest exponential relationships between crossing number and the proportions of both prime alternating knots and links with $\gamma>\Gamma$.  

\begin{conjecture}
    Consider, among all prime alternating knots, the proportion $\varepsilon_K(c)$ that have crosscap number greater than unoriented genus, computed by crossing number $c$.  Similarly, let $\varepsilon_L(c)$ be this proportion for prime alternating links.
    Then, both proportions exponentially decay with respect to crossing number, i.e., 
    $$\varepsilon_K(c) \propto e^{-c\lambda_K } \qquad and \qquad \varepsilon_L(c) \propto e^{-c\lambda_L },$$
    for positive constants $\lambda_K,\lambda_L$.
    In particular, both proportions approach zero as the crossing number goes to infinity.  
    \label{C:LogFreq}
\end{conjecture}

A linear regression of our data in Table~\ref{T:DefectProp} calculates $\varepsilon_K(c) \approx 5.25e^{-0.545c}$ for knots (using  $7\leq c \leq 19$) and $\varepsilon_L(c) \approx 4.86e^{-0.461c}$ for links (using $8\leq c \leq 14$).  (No examples of $\gamma<\Gamma$ occur for prime alternating knots with $c<7$ crossings. The only four examples that occur for prime alternating links with $c<8$ crossings are described in Conjecture~\ref{C:gamma3ex}.)  

Theorem 5.2 of \cite{cklsv} is a statement analogous to Conjecture \ref{C:LogFreq}, but for 2-bridge knots. It provides exact formulas for $\varepsilon_K(c)$ in the 2-bridge case for every crossing number $c$.

\begin{table}[H]
        \begin{center}
    \small
    \begin{tabular}{| c |cccc| c c c|}
\hline\hline
\diagbox{$c$}{$\gamma$} & 3 & 5 & 7 & 9 & total & \# of knots & proportion \\
\hline\hline
$3$ & $ $ & $  $ & $  $ & $  $ & $ 0 $ & $ 1 $ & $ 0$\\ \hline
$4$ & $ $ & $  $ & $  $ & $  $ & $ 0 $ & $ 1 $ & $ 0$\\ \hline
$5$ & $ $ & $  $ & $  $ & $  $ & $ 0 $ & $ 2 $ & $ 0$\\ \hline
$6$ & $ $ & $  $ & $  $ & $  $ & $ 0 $ & $ 3 $ & $ 0$\\ \hline
$7$ & $1 $ & $  $ & $  $ & $  $ & $ 1 $ & $ 7 $ & $ 0.14286$\\ \hline
$8$ & $1 $ & $  $ & $  $ & $  $ & $ 1 $ & $ 18 $ & $ 0.05556$\\ \hline
$9$ & $2 $ & $  $ & $  $ & $  $ & $ 2 $ & $ 41 $ & $ 0.04878$\\ \hline
$10$ & $1 $ & $ 1 $ & $  $ & $  $ & $ 2 $ & $ 123 $ & $ 0.01626$\\ \hline
$11$ & $3 $ & $ 2 $ & $  $ & $  $ & $ 5 $ & $ 367 $ & $ 0.01362$\\ \hline
$12$ & $2 $ & $ 8 $ & $  $ & $  $ & $ 10 $ & $ 1288 $ & $ 0.00776$\\ \hline
$13$ & $4 $ & $ 19 $ & $  $ & $  $ & $ 23 $ & $ 4878 $ & $ 0.00472$\\ \hline
$14$ & $2 $ & $ 42 $ & $ 3 $ & $  $ & $ 47 $ & $ 19536 $ & $ 0.00241$\\ \hline
$15$ & $6 $ & $ 85 $ & $ 25 $ & $  $ & $ 116 $ & $ 85263 $ & $ 0.00136$\\ \hline
$16$ & $3 $ & $ 150 $ & $ 132 $ & $  $ & $ 285 $ & $ 379799 $ & $ 0.00075$\\ \hline
$17$ & $7 $ & $ 278 $ & $ 586 $ & $  $ & $ 871 $ & $ 1769979 $ & $ 0.00049$\\ \hline
$18$ & $3 $ & $ 435 $ & $ 2069 $ & $ 42 $ & $ 2549 $ & $ 8400285 $ & $ 0.00030$\\ \hline
$19$ & $9 $ & $ 742 $ & $ 6395 $ & $ 458 $ & $ 7604 $ & $ 40619385 $ & $ 0.00019$\\ \hline
\hline
    \end{tabular}
    \caption{The number and proportion of prime alternating knots with $c\leq 19$ crossings and $\gamma>\Gamma$}
    \label{T:DefectK}
    \end{center}
\end{table}

\begin{table}[H]
        \begin{center}
    \small
    \begin{tabular}{| c |cccccc| c c c|}
\hline\hline
\diagbox{$c$}{$\gamma$} & 2 & 3 & 4 & 5 & 6 & 7 & total & \# of links & proportion \\
\hline\hline
$2$ & $1 $ & $  $ & $  $ & $  $ & $  $ & $  $ & $ 1 $ & $ 1 $ & $ 1.0$\\ \hline
$4$ & $ 1$ & $  $ & $  $ & $  $ & $  $ & $  $ & $ 1 $ & $ 1 $ & $ 1.0$\\ \hline
$5$ & $ $ & $  $ & $  $ & $  $ & $  $ & $  $ & $ 0 $ & $ 1 $ & $ 0$\\ \hline
$6$ & $1 $ & $ 1 $ & $  $ & $  $ & $  $ & $  $ & $ 2 $ & $ 5 $ & $ 0.4$\\ \hline
$7$ & $ $ & $  $ & $  $ & $  $ & $  $ & $  $ & $ 0 $ & $ 7 $ & $ 0 $\\ \hline
$8$ & $ 1$ & $ 1 $ & $ 2 $ & $  $ & $  $ & $  $ & $ 4 $ & $ 21 $ & $ 0.19048 $\\ \hline
$9$ & $ $ & $  $ & $ 2 $ & $  $ & $  $ & $  $ & $ 2 $ & $ 55 $ & $ 0.03636 $\\ \hline
$10$ & $1 $ & $ 2 $ & $ 7 $ & $ 2 $ & $  $ & $  $ & $ 12 $ & $ 174 $ & $ 0.06897 $\\ \hline
$11$ & $ $ & $  $ & $ 8 $ & $ 5 $ & $  $ & $  $ & $ 13 $ & $ 548 $ & $ 0.02372 $\\ \hline
$12$ & $ 1$ & $ 3 $ & $ 21 $ & $ 18 $ & $ 9 $ & $  $ & $ 52 $ & $ 2020 $ & $ 0.02574 $\\ \hline
$13$ & $ $ & $ 22 $ & $ 34 $ & $ 22 $ & $  $ & $  $ & $ 78 $ & $ 7539	$ & $	0.01035$\\ \hline
$14$ & $ 1$ & $ 4 $ & $ 43 $ & $ 84 $ & $ 117 $ & $ 15 $ & $ 264 $ & $ 31811 $ & $ 0.00830$\\ \hline
\hline
    \end{tabular}
    \caption{The number and proportion of prime alternating links (with multiple components) with $c\leq 14$ crossings and $\gamma>\Gamma$}
    \label{T:DefectL}
    \end{center}
\end{table}

\begin{conjecture}
    For each even crossing number $2k\geq 6$, there are exactly $p(k,3) = \left\lfloor \frac{k^{2} + 3}{12} \right\rfloor$
    prime alternating links $L$ with $\Gamma(L)<\gamma(L)=3$, namely the pretzel links $P(2\ell,2m,2n)$ with $0<\ell\leq m\leq n$ and $\ell+m+n=k$.  Here, $p(k,3)$ denotes the number of partitions of $k$ into three (positive) parts \cite{oeis}.
    \label{C:gamma3ex}
\end{conjecture}

Regarding Conjecture \ref{C:gamma3ex}, recall that, by convention, we assume that links have more than one component, we consider a link and its mirror image to be equivalent, and $\gamma$ and $\Gamma$ are equal for mirror images. 

We offer one more conjecture, which is motivated as follows.  According to Theorem \ref{T:AK}, a knot or link $K$ with an alternating diagram $D$ satisfies $\gamma(K)>\Gamma(K)$ if and only if (1) there is a Kauffman state $x$ of $D$ which has strictly more state circles than all other states of $D$, and (2) the state surface from $x$ is 2-sided. The number of 2-sided state surfaces of $D$ equals $2^{|K|}$, and so, heuristically, one would expect a random $n$-component, $c$-crossing alternating link to be $2^{n-1}$ times more likely than a random $c$-crossing alternating knot to have $\gamma>\Gamma$. On the other hand, one would expect some noise in the data to make this phenomenon emerge in a stable way only for sufficiently large crossing numbers.  Indeed, this is what we find in the data, with the curious addendum that the phenomenon emerges earlier when crossing numbers are even than it does when crossing numbers are odd.

\begin{table}[H]
\begin{center}
    \includegraphics[width=4in]{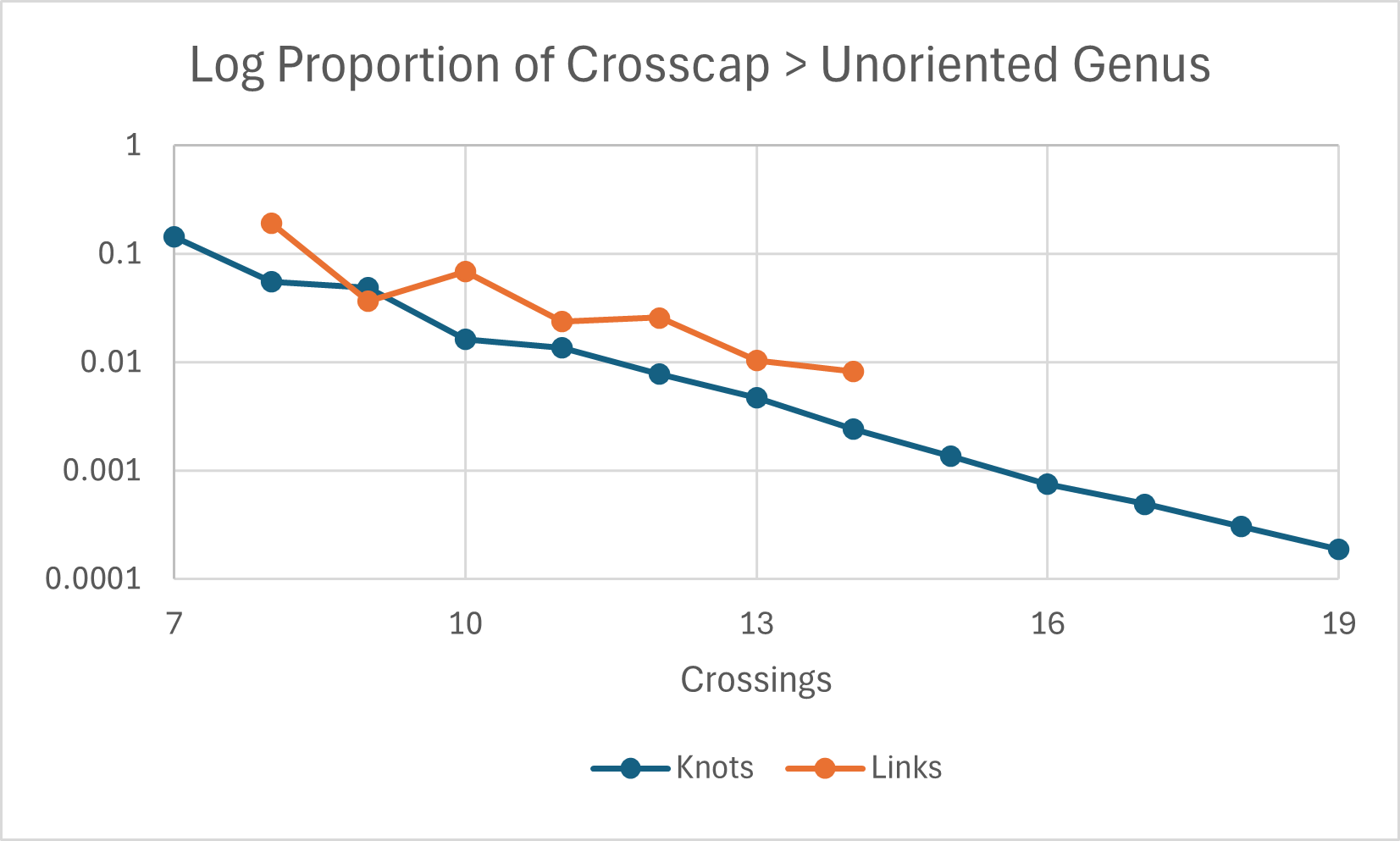} \\
    \caption{The proportions of prime alternating knots and of links with crosscap number greater than unoriented genus indicates a log-linear relationship with crossing number.  See Conjecture~\ref{C:LogFreq}.}
    \label{T:DefectProp}
    \end{center}
    \end{table}

\begin{conjecture}
    For all even $c\geq 4$ and all odd $c\geq 11$, the proportion of $c$-crossing prime alternating links $L$ with $\gamma(L)>\Gamma(L)$ is larger than the proportion of $c$-crossing prime alternating knots $K$ with $\gamma(K)>\Gamma(K)$. 
\end{conjecture}

We remind the reader that our data and code will be available at 
\url{https://github.com/IsaiasBahena/minimal-genus-state-surfaces}.

\end{document}